\documentclass[journal,transmag]{IEEEtran}
\usepackage{amssymb}
\usepackage{color}
\usepackage[T1]{fontenc}
\usepackage{graphicx}
\usepackage{amsmath}
\usepackage{amsfonts}
\usepackage{mathdots}
\usepackage{subfigure}
\usepackage{comment}
\usepackage{calrsfs}
\usepackage{pgfplots}
\pgfplotsset{compat=1.16}
\usepackage{tikz}
\usetikzlibrary{external,calc,shapes,arrows,patterns,decorations.pathmorphing,decorations.markings,hobby,spy}
\newcommand{\norm}[1]{\left \vert\left \vert #1 \right \vert\right \vert}
\newcommand{\maxs}[1]{\max \left\lbrace #1 \right\rbrace}
\newcommand{\mins}[1]{\min \left\lbrace #1 \right\rbrace}

\newcommand{\modi}[1]{\text{mod}\!\left( #1 \right)}

\newcommand{\floor}[1]{\pmb{\left\lfloor\vphantom{#1}\right.}\! #1 \! \pmb{\left.\vphantom{#1}\right\rfloor}}
\DeclareMathOperator*{\ceil}{ceil}

\DeclareMathOperator*{\sign}{sign}

\DeclareMathAlphabet{\pazocal}{OMS}{zplm}{m}{n}
\definecolor{cblue}{RGB}{0,114,189}
\definecolor{corange}{RGB}{236,176,32}
\definecolor{cgreen}{RGB}{0,127,0}
\definecolor{cred}{RGB}{217,87,23}

\newtheorem{assumption}{Assumption}
\newtheorem{definition}{Definition}

\newtheorem{proof}{Proof}
\newtheorem{lemma}{Lemma}
\newtheorem{theorem}{Theorem}
\newtheorem{remark}{Remark}
\definecolor{lgray}{gray}{0.95}
\definecolor{lred}{RGB}{255,213,213}
\definecolor{RoyalRed}{RGB}{157,16, 45}
\definecolor{darkblue}{RGB}{0,71,140}
\definecolor{cblue}{RGB}{0,114,189}
\definecolor{corange}{RGB}{236,176,32}
\definecolor{blues}{RGB}{19,127,195}
\definecolor{reds}{RGB}{237,192,186}
\definecolor{cred}{RGB}{217,87,23}
\definecolor{reds2}{RGB}{227,154,145}
\definecolor{cgrey}{RGB}{200,200,200}
\definecolor{grey}{RGB}{160,160,160}
\definecolor{ccyan}{RGB}{173,235,255}
\definecolor{cgreen}{RGB}{0,127,0}
\definecolor{bluess}{RGB}{86,165,184}
\definecolor{greens}{RGB}{177,217,177}
\definecolor{yellows}{RGB}{207,197,99}
\newcommand{\full}[1]{\textcolor{#1}{\protect\tikz[baseline]{\protect\draw[line width=0.3mm] (0,.6ex)--++(0.5,0);}}}

\ifCLASSINFOpdf
\else
\fi
\begin{document}

\title{Extremum Seeking with Intermittent Measurements:\\ A Lie-brackets Approach}


\author{\IEEEauthorblockN{Christophe Labar\IEEEauthorrefmark{1,}\IEEEauthorrefmark{2},
Christian Ebenbauer\IEEEauthorrefmark{2} and
Lorenzo Marconi\IEEEauthorrefmark{1}}
\IEEEauthorblockA{\IEEEauthorrefmark{1} C.A.SY.-DEI, University of Bologna, 40126 Bologna, Italy}
\IEEEauthorblockA{\IEEEauthorrefmark{2} Chair of Intelligent Control System, RWTH Aachen University,  52074 Aachen, Germany}
\thanks{This is an extended version of the paper submitted to IEEE-TAC. This research was supported by the European Project ”AerIal RoBotic technologies for professiOnal seaRch aNd rescuE” (AirBorne), Call: H2020, ICT-25-2016/17, Grant Agreement no: 780960 and the Deutsche Forschungsge-meinschaft, Grant Agreement no: EB 425/8-1. 
Corresponding author: C. Labar (email: christophe.labar@ic.rwth-aachen.de).
}}

%



\IEEEtitleabstractindextext{%
\begin{abstract}
Extremum seeking systems are powerful methods able to steer the input of a (dynamical) cost function towards an optimizer, without any prior knowledge of the cost function. To achieve their objective, they typically combine time-periodic signals with the on-line measurement of the cost. However, in some practical applications, the cost can only be measured during some regular time-intervals, and not continuously, contravening the classical extremum seeking framework. In this paper, we first analyze how existing Lie-bracket based extremum seeking systems behave when being fed with intermittent measurements, instead of continuous ones. We then propose two modifications of those schemes to improve both the convergence time and the steady-state accuracy in presence of intermittent measurements. The performances of the different schemes are compared on a case study.
\end{abstract}

}

\maketitle

\IEEEdisplaynontitleabstractindextext

%
\IEEEpeerreviewmaketitle

\section{Introduction}
Extremum seeking systems are powerful methods to steer the output of a dynamical system towards the optimizer of an associated cost function. To achieve their objective, they do not need any prior knowledge of the mathematical expression of the cost function, or the value of its gradient. They solely rely on the on-line measurement of the cost. 

Different techniques have been proposed to estimate an ascent (or descent) direction from the on-line measurement of the cost (see e.g. \cite{Att19}, \cite{Hun14}, \cite{Gua17}, or \cite{Tol17}). In both the cases of averaging and Lie-bracket based extremum seeking, this estimation is obtained by combining the on-line measurement of the cost with time-periodic signals. Depending on this combination, the extremum seeking system may approximate different continuous-time optimization laws, like gradient-descent laws (see e.g. \cite{Lab18}, \cite{Sut19}, or \cite{Gru20}) and Newton-based optimization laws (see e.g. \cite{Lab19}, \cite{Moa10}, or \cite{Gro16}).

Those algorithms have numerous application domains, ranging from the maximization of the power produced by wind turbines via the generator speed (see e.g. \cite{Gha14}, \cite{Rot17} or \cite{Raf18}), to the optimization of the production yield in bio-processes (see e.g. \cite{Bas09}, \cite{Hal19} or \cite{Dew17}), the maximization of the energy efficiency in vapor compression systems (see e.g. \cite{Gua14} or \cite{Koe14}), or the source seeking, when the source emits a scalar signal achieving an optimum at its position (see e.g. \cite{Man20} or \cite{Xu19}). 

Typically, those works assume that the cost may be continuously measured. However, in some cases, only intermittent measurements are available. Namely, the cost can only be measured during some time-intervals. This is notably the case if one considers the application of search and rescue of avalanche victims (see e.g. \cite{Ila20} and \cite{Sil17}), that has motivated the present work. Avalanche beacons used in this context emit an electromagnetic field whose intensity decreases as one gets farther away from them. Since the intensity achieves its maximum at the victim location, finding the victim is equivalent to localizing the maximal field intensity. However, in order to save their batteries and, hence, increase the total time of transmission, the beacons generate the electromagnetic field by pulses of a few tenths of second every second. Following the European standard DIN EN 300718-1, the pulses duration, the break-time, and the total period, should be at least 70ms, 400ms and 1000ms$\pm$300ms, respectively.

Since the on-line measurement of the cost is the only information used by extremum seeking schemes, one may expect their performances to be impacted in case of intermittent measurements. 

Motivated by this application, in this paper, we first analyze how existing extremum seeking algorithms behave when being fed with intermittent measurements, instead of continuous ones. We then propose two main modifications to improve both the steady-state and convergence properties in presence of intermittent measurements.

The rest of the paper is structured as follows. In Section \ref{sec:preliminaries}, the notations and definitions used along the paper, together with the Lie-bracket approximation, are introduced. The considered problem is formally stated in Section \ref{sec:problem}. The main results are presented in Section \ref{sec:main_results}, and the performances obtained with the different schemes are compared through a running case study.
\section{Preliminaries}
\label{sec:preliminaries}
\subsection{Notations and Definitions}
The Euclidean norm of a vector $x\in\mathbb{R}^n$ is denoted by $\norm{x}$. We use $\mathbb{R}_{>0}$, $\mathbb{Q}_{>0}$, and $\mathbb{N}_{>0}$, for the sets of strictly positive real numbers, strictly positive rational numbers, and strictly positive natural numbers, respectively. The gradient and Hessian matrix of a sufficiently differentiable function $h:\mathbb{R}^n\rightarrow\mathbb{R}$ are denoted by $\nabla h(x)$ and $\nabla^2 h(x)$, respectively. The Jacobian of a map $f:\mathbb{R}^n\rightarrow\mathbb{R}^m$ is denoted by $\pazocal{D} f(x)$. Let $f:\mathbb{R}^n\rightarrow\mathbb{R}^n$ and $g:\mathbb{R}^n\rightarrow\mathbb{R}^n$ be two differentiable vector fields. We define the Lie-derivative of $f(x)$ with respect to $g(x)$ by $\pazocal{L}_g f(x)=\pazocal{D} f(x) g(x)$, and the Lie-bracket between $f(x)$ and $g(x)$ by $[f,g](x)=\pazocal{L}_f g(x)-\pazocal{L}_g f(x)$. The remainder of the Euclidean division of $a\in\mathbb{R}$ by $b\in\mathbb{R}\backslash\{0\}$ is denoted by $\modi{a,b}$. We use $\textbf{LCM}(a_1,a_2,...,a_n)$ for the least common multiple of $\{a_1,a_2,...,a_n\}$, with $a_i\in\mathbb{Q}$, for $i\in\{1,2,...,n\}$. The $\delta$-neighborhood of a point $x^*\in\mathbb{R}^n$, with $\delta\in\mathbb{R}_{>0}$, is defined by $U_{\mathbb{R}^n}^{x^*}(\delta)=\{x\in\mathbb{R}^n:\norm{x-x^*}\leq \delta\}$. The function $\sign:\mathbb{R}\rightarrow\{-1,0,1\}$ is defined as $\sign(a)=-1$ if $a<0$, $\sign(a)=0$ if $a=0$, and $\sign(a)=1$ if $a>0$. 

Along the paper, we will also refer to the following notion of semi-global practical uniform asymptotic stability:
\begin{definition}
	\label{def:sgpuas}
	The origin is said to be semi-globally practically uniformly asymptotically stable (sGPUAS) for the system $\dot{x}(t)=f(t,x(t),\epsilon)$, with the vector of parameters $\epsilon\in\mathbb{R}^{n_\epsilon}$, if the following holds. For every $\delta_B\in\mathbb{R}_{>0}$, and $\delta_V\in\mathbb{R}_{>0}$, there exist a $\delta_Q\in\mathbb{R}_{>0}$, a $\delta_W\in\mathbb{R}_{>0}$, and an $\epsilon^*_1\in\mathbb{R}_{>0}$, such that, for all $\epsilon_1\in(0,\epsilon_1^*)$, there exists an $\epsilon^*_2\in\mathbb{R}_{>0}$ such that, for all $\epsilon_2\in(0,\epsilon_2^*)$, ..., there exists an $\epsilon_{n_\epsilon}^*\in\mathbb{R}_{>0}$ such that, for all $\epsilon_{n_\epsilon}\in(0,\epsilon_{n_\epsilon}^*)$, there exists a $t_1\in\mathbb{R}_{>0}$ such that, for all $t_0\in\mathbb{R}$:
	\begin{enumerate}
		\item if $x(t_0)\!\in\!U_{\mathbb{R}^n}^{x^*}(\delta_V)$, then $x(t)\!\in\! U_{\mathbb{R}^n}^{x^*}(\delta_W)$, for all $t\!\geq\!t_0$;
		\item if $x(t_0)\in U_{\mathbb{R}^n}^{x^*}(\delta_Q)$, then $x(t)\in U_{\mathbb{R}^n}^{x^*}(\delta_B)$, for all $t\geq t_0$;
		\item if $x(t_0)\in U_{\mathbb{R}^n}^{x^*}(\delta_V)$, then $x(t)\in U_{\mathbb{R}^n}^{x^*}(\delta_B)$, for all $t\geq t_0+t_1$.
	\end{enumerate}
\end{definition}

Occasionally, Definition \ref{def:sgpuas} is also used for time-delayed systems in the form $\dot{x}(t)\!=\!f(t,x_{[t-T,t]},\epsilon)$, with $T\in\mathbb{R}_{\geq 0}$ and $x_{[t-T,t]}$ the $x$-trajectory on the interval $[t-T,t]$. In that case, its three conditions are supposed to hold with $x_{[t_0-T,t_0]}\in\!U_{\mathbb{R}^n}^{x^*}(\delta_V)$ and $x_{[t_0-T,t_0]}\in\!U_{\mathbb{R}^n}^{x^*}(\delta_B)$, instead of $x(t_0)\!\in\!U_{\mathbb{R}^n}^{x^*}(\delta_V)$ and $x(t_0)\!\in\!U_{\mathbb{R}^n}^{x^*}(\delta_B)$.

Note that, if the value of $\epsilon$ in Definition \ref{def:sgpuas} does not depend on $\delta_B$ and $\delta_V$, and if it also holds $\lim_{t\rightarrow \infty} \norm{x(t)}=0$, then sGPUAS reduces to global uniform asymptotic stability.

\subsection{Lie-bracket Approximation}
\label{sec:LB}
In this section, we introduce the Lie-bracket approximation, that will serve us as a basis in the sequel.

Consider the input-affine system
\begin{equation}
\dot{x}(t)=\sqrt{\omega} \sum_{i=1}^{l} f_i(x(t)\!) u_i(k_i\omega t),
\label{eq:def_input_affine}
\end{equation}
with $x(t)\in\mathbb{R}^n$ the state vector, $u_i(t)\in\mathbb{R}$ the control inputs, $f_i\!:\!\mathbb{R}^n\rightarrow\mathbb{R}^n$ the system vector fields, $k_i\!\in\!\mathbb{Q}_{>0}$, and $\omega\!\in\!\mathbb{R}_{>0}$.

Suppose that the vector fields $f_i$ and the control inputs $u_i$ satisfy the following two assumptions:
\begin{assumption}
	\label{as:dithers}
	For all $i\in\{1,2,...,l\}$, the control input $u_i:\mathbb{R}\rightarrow\mathbb{R}$
	\begin{enumerate}
		\item is a measurable function such that $\sup_{t\in\mathbb{R}} \vert u_i(t) \vert \leq 1$;
		\item is $2\pi$-periodic, i.e. $u_i(t+2\pi)=u_i(t)$, for all $t\in\mathbb{R}$;
		\item has zero mean on a period, i.e. $\displaystyle\int_{0}^{2\pi} u_i(\tau)d\tau=0$.
	\end{enumerate}
\end{assumption}
\begin{assumption}
	\label{as:vector_fields}
	For all $i\in\{1,2,...,l\}$, the vector field $f_i:\mathbb{R}^n\rightarrow\mathbb{R}^n$ is of class $C^2$.
\end{assumption}

Then, the following time-invariant system, called \textit{Lie-bracket system}, may be associated with system \eqref{eq:def_input_affine}
\begin{equation}
\dot{\overline{x}}=\sum_{\substack{1\leq i< l\\ i< j\leq l}} [f_i,f_j](\overline{x})\gamma_{ij},
\label{eq:def_lie_bracket}
\end{equation}
where we introduced
\begin{equation}
\gamma_{ij}:=\frac{\omega}{T}\int_0^T \int_0^{\theta} u_j(k_j\omega \theta)u_i(k_i\omega \tau)\, d\tau\, d\theta,
\label{eq:def_gamma_ij}
\end{equation}
with $T=2\pi\omega^{-1}\textbf{LCM}\left(k_1^{-1},k_2^{-1},...,k_l^{-1}\right)$.  

It may be shown that the trajectory of system \eqref{eq:def_input_affine} approximates the one of the Lie-bracket system \eqref{eq:def_lie_bracket}, with an accuracy that may be made arbitrarily large by increasing the value of $\omega$ (see e.g. Theorem 3.1.1. in \cite{Dur_thesis}). 

A direct consequence of this closeness of trajectories is the following stability result (see e.g. Lemma 3.1 in \cite{Lab18}):
\begin{lemma}
	\label{lem:approx_traj}
	Consider system \eqref{eq:def_input_affine}, and let Assumptions \ref{as:dithers}-\ref{as:vector_fields} hold. Suppose furthermore that the origin is globally asymptotically stable for the Lie-bracket system \eqref{eq:def_lie_bracket}. Then, the origin is sGPUAS for system \eqref{eq:def_input_affine}, with the parameter $\omega^{-1}$.
\end{lemma}

It is thus possible to approximate the behavior of a system, that involves the Jacobian of some vector fields, by implementing a system that only involves those vector fields. This allows, for instance, to approximate the behavior of a gradient-descent law, by implementing a system that only uses the on-line value of the cost. This strategy may thus be exploited to design extremum seeking systems.

\section{Problem Statement}
\label{sec:problem}
Consider the system
\begin{equation}
\dot{x}=v,
\label{eq:def_system}
\end{equation}
where $x\in\mathbb{R}^n$ is the state vector and $v\in\mathbb{R}^n$ the control input. 

In this work, we aim at designing a control law $v(t)$ that steers system \eqref{eq:def_system} to the minimizer of a cost function $h(x)$. We assume that this cost function has the following properties:

\begin{assumption}
	\label{as:cost}
	The cost function $h:\mathbb{R}^n\rightarrow\mathbb{R}$ belongs to the class $C^2$. Furthermore, there exists an $x^*\in\mathbb{R}^n$ such that $\nabla^T h(x)(x-x^*)>0$, for all $x\in\mathbb{R}^n\backslash\{x^*\}$.
\end{assumption}

\begin{remark}
	Assumption \ref{as:cost} ensures the existence of a unique minimizer $x^*\in\mathbb{R}^n$ for the cost $h(x)$. It also guarantees that $h(x)$ does not have any other stationary point than $x^*$.
\end{remark}

To achieve our objective, neither the mathematical expression of the cost function $h(x)$, nor the value of its gradient $\nabla h(x)$, are known. The only information available is the intermittent measurement of $h(x)$. More precisely, we only know $h_m(t,x(t)\!)$, for all $t\in\mathbb{R}$, with
\begin{equation}
h_m(t,x(t)\!)=\left\lbrace\begin{matrix}
h(x(t)\!) &\text{ if } \modi{t,T_s}\in[0,\epsilon)\\
0 &\text{ if } \modi{t,T_s}\in[\epsilon,T_s),
\end{matrix}\right.
\label{eq:def_hm}
\end{equation}
where $T_s\in\mathbb{R}_{>0}$, and $\epsilon\in(0,T_s]$, are unknown parameters. Those two parameters are solely determined by the transmitter (and receiver) characteristics. An illustration of $h_m(t,x(t)\!)$ and $h(x(t)\!)$ is presented in Figure \ref{fig:intermittent_meas}.

Note that the problem addressed in this work contrasts from the classical continuous-time extremum seeking framework, where $h(x(t)\!)$ is known for all $t\in\mathbb{R}$ (i.e. $h_m(t,x(t)\!)=h(x(t)\!)$, for all $t\in\mathbb{R}$). It also differs from the discrete-time, sample-data, and networked extremum seeking frameworks (see e.g. \cite{Haz19}, \cite{Kho13}, and \cite{Pre20}), where only one measurement of the cost is used on a sampling period. Furthermore, since neither $T_s$ nor $\epsilon$ are known, the selection of an adequate sampling time would not be straightforward.

\begin{figure}
	\centering
	\begin{tikzpicture}
	\begin{axis}[width=0.45\textwidth,height=4cm,x axis line style={-stealth},y axis line style={-stealth},xtick={0,0.2,1,1.2,2,2.2,3},xticklabels={0,$\epsilon$, \hspace{-1mm}$T_s$, \hspace{6mm}$T_s\!+\!\epsilon$, \hspace{-3mm}$2T_s$, \hspace{8mm}$2T_s\!+\!\epsilon$,$3T_s$},ytick={0,2,4},yticklabels={0,2,4},ymax = 5,ymin=-0.4,xmax=3.01,xmin=-0.1,xlabel={Time [$s$]},
	ylabel={Cost [$/$]}]
	\addplot[color=cblue,smooth] coordinates {(-0.1,0.5) (-0.05,0.6) (0,1) (0.05,1.5) (0.1,1.7) (0.15, 2.5) (0.2,3) (0.25,3.1) (0.35,3.7) (0.7,3) (1,3.5) (1.05,3.6) (1.1,3.8) (1.15,3.7 ) (1.2,4) (1.25, 4.1) (1.5,4.5) (1.7,4.3) (1.95, 3.4) (2,3.2) (2.05,2.8) (2.1,2.4) (2.15,2) (2.2,2.2) (2.25,2.3) (2.5,3) (2.7,3.2) (3,4)};
	\addplot[color=corange,smooth, line width=0.45mm] coordinates {(0,1) (0.05,1.5) (0.1,1.7) (0.15, 2.5) (0.2,3)};
	\addplot[color=corange,line width=0.45mm] coordinates {(-0.1,0) (0,0)};
	\addplot[color=corange,smooth,line width=0.45mm] coordinates {(1,3.5) (1.05,3.6) (1.1,3.8) (1.15,3.7 ) (1.2,4)};
	\addplot[color=corange,smooth,line width=0.45mm] coordinates {(2,3.2) (2.05,2.8) (2.1,2.4) (2.15,2) (2.2,2.2)};
	\addplot[color=corange,line width=0.45mm] coordinates {(0.2,0) (1,0)};
	\addplot[color=corange,line width=0.45mm] coordinates {(1.2,0) (2,0)};
	\addplot[color=corange,line width=0.45mm] coordinates {(2.2,0) (3,0)};
	\end{axis}
	\end{tikzpicture}
	\caption{Illustration of the intermittent measurements: continuous value of the cost $h(x(t)\!)$ (\full{cblue}) and intermittent measurement $h_m(t,x(t)\!)$ (\full{corange}).}
	\label{fig:intermittent_meas}
\end{figure}
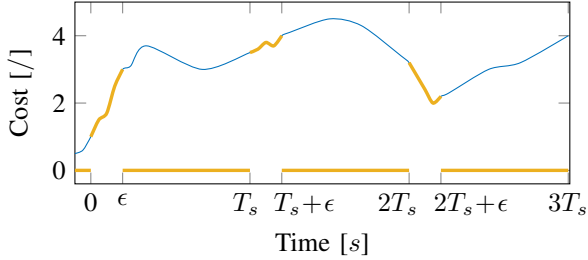
\section{Main Results}
\label{sec:main_results}
To address the problem, we first analyze whether existing extremum seeking systems are able to handle intermittent measurements of the cost. We then modify those schemes to make them tailored to the use of intermittent measurements.
\subsection{Analysis of Classical Extremum Seeking Systems}
\label{sec:basic_scheme}
To start with, we examine extremum seeking systems that may be written in the following form
\begin{equation}
\dot{x}_c(t)=\sqrt{\omega} \sum_{i=1}^{l} f_i(h(x_c(t)\!)\!) u_i(k_i\omega t),
\label{eq:classical_ES}
\end{equation}
where the dithers $u_i(t)$ and the cost $h(x)$ obey Assumptions \ref{as:dithers} and \ref{as:cost}, respectively, and the maps $f_i(y)$ are such that:
\begin{assumption}
	\label{as:vector_fields_gen}
	For all $i\in\{1,2,...,l\}$, the map $f_i:\mathbb{R}\rightarrow\mathbb{R}^n$ is of class $C^2$. Furthermore, for all $y\in\mathbb{R}$, it holds  $\sum_{\genfrac{}{}{0pt}{2}{1\leq i< l}{i< j\leq l}} \left(\pazocal{D} f_j(y) f_i^T(y)-\pazocal{D} f_i(y) f_j^T(y) \right)\gamma_{ij}=-\rho I_{n\times n}$, with $\rho\in\mathbb{R}_{>0}$, and $\gamma_{ij}$ defined in \eqref{eq:def_gamma_ij}.
\end{assumption}

\begin{remark}
	In virtue of Assumption \ref{as:vector_fields_gen}, the Lie-bracket system associated with \eqref{eq:classical_ES} is the gradient-descent law $\dot{\overline{x}}=-\rho\nabla h(\overline{x})$ (see Section \ref{sec:LB}). Referring to Lemma \ref{lem:approx_traj}, and given Assumption \ref{as:cost}, the minimizer of $h(x)$ is thus sGPUAS for the extremum seeking system \eqref{eq:classical_ES}, with the parameter $\omega^{-1}$. Note that Assumption \ref{as:vector_fields_gen} is fulfilled if the vector fields are selected as in \cite{Gru18} (see Theorems 1 and 2 with $F_{0i}\!=\!\rho \gamma_{1i,2i}^{-1}$), which covers the choices in \cite{Dur13} and \cite{Sch14}.
\end{remark}

To get an insight into the behavior of system \eqref{eq:classical_ES} fed with intermittent measurements, instead of continuous ones, namely
\begin{equation}
\dot{x}(t)=\sqrt{\omega} \sum_{i=1}^{l} f_i(h_m(t,x(t)\!)\!) u_i(k_i\omega t),
\label{eq:classical_source}
\end{equation}
with $h_m(t,x)$ defined in \eqref{eq:def_hm}, we consider the following case study: $h(x)=(x-2)^2+10$, $f_1(h(x)\!)=2\rho h(x)$, $f_2(h(x)\!)=1$, $u_1(t)=\cos(t)$, $u_2(t)=\sin(t)$, $k_1=k_2=1$, $\rho=0.25s^{-1}$, $t_0=0s$, and $x(t_0)=-1$.  

For the sake of comparison, the trajectory obtained with continuous measurements is shown in Figure \ref{fig:Krstic_continuous} for $\omega=20\,\pi$ rad/s. In agreement with the intuition, it may be observed in Figures \ref{fig:Krstic_intermittent} and \ref{fig:Krstic_intermittent_unstable} that the value of $\epsilon$ strongly influences the stability properties of \eqref{eq:classical_source}. While the practical convergence to $x^*$ is preserved for $\epsilon=0.1s$ (see Figures \ref{fig:Krstic_continuous} and \ref{fig:Krstic_intermittent}), the trajectory diverges for $\epsilon=0.17s$. This may be explained as follows. For $\epsilon=0.1s$, the time during which the cost measurement is not available is a multiple of the dithers period. Since the dithers have zero mean (see Assumption \ref{as:dithers}), the system state is the same at the beginning and at end of the transmission break. Accordingly, since the dithers have made a integer number of periods, the "averaging" may resume as if there had not been any transmission break. On the other hand, for $\epsilon=0.17s$ (see Figure \ref{fig:Krstic_intermittent_unstable}), the presence of an incomplete dithers period during the transmission break prevents the "averaging" to resume as if there had not been any transmission break. Therefore, both the incomplete dithers periods at the end of the transmission time and break may have a negative impact on the dynamics. More precisely, one may see in Figure \ref{fig:Krstic_intermittent_unstable} that, after the first dithers period (i.e. for $t=0.1s$), it holds $\norm{x(T)-x^*}<\norm{x(0)-x^*}$. However, since $\epsilon$ is smaller than two dithers periods, the "averaging" of the second period cannot be completed, and it results $\norm{x(2T)-x^*}>\norm{x(0)-x^*}$. This deviation of $x$ from $x^*$ is not compensated by the incomplete dithers period during the transmission break, so that $\norm{x(T_s)-x^*}>\norm{x(0)-x^*}$. The same phenomenon occurs during the following periods, leading to the divergence of the trajectory. Fortunately, by increasing the dithers frequency, the impact of the incomplete dithers periods may be made arbitrarily small. This may be observed in Figure \ref{fig:Krstic_intermittent_ok} where $\omega$ was increased to $2002\,\pi$ rad/s, while keeping $\epsilon=0.17s$.

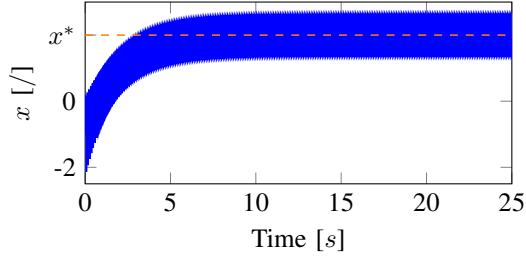
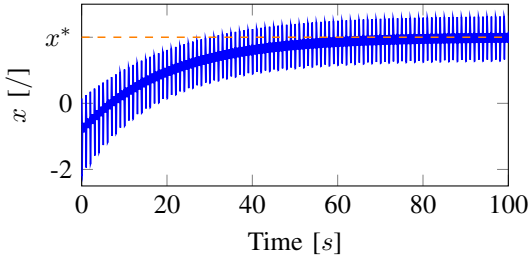
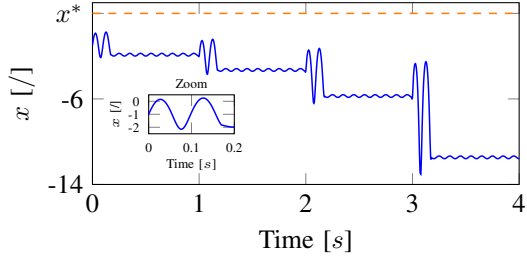
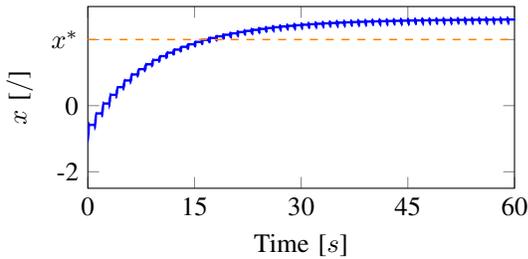
\begin{figure}[!ht]
	\centering
	\subfigure[Trajectory obtained with continuous measurements: $T_s=\epsilon$ and $\omega=20\,\pi\, rad/s$]{
		\begin{tikzpicture}
		\begin{axis}[width=0.4\textwidth,height=4cm,x axis line style={-stealth},y axis line style={-stealth},xtick={0,5,10,15,20,25},xticklabels={0,5,10,15,20,25},ytick={-2,0,2},yticklabels={-2,0,$x^*$}, ymax = 3,ymin=-2.5,xmax=25,xmin=0,xlabel={Time [$s$]},ylabel={$x$ [$/$]},ylabel shift=-2pt]
		\addplot[smooth,color=blue, line width=0.2mm] table[x index=0, y index=1] {Pictures/Krstic_on.dat};
		\addplot[smooth,color=orange, line width=0.2mm, dashed] (0,2)--(25,2);
		\end{axis}
		\end{tikzpicture}
		\label{fig:Krstic_continuous}}
	\subfigure[Trajectory obtained with intermittent measurements: $T_s=1s$, $\epsilon=0.1s$, and $\omega=20\,\pi\, rad/s$]{
		\begin{tikzpicture}
		\begin{axis}[width=0.4\textwidth,height=4cm,x axis line style={-stealth},y axis line style={-stealth},xtick={0,20,40,60,80,100},xticklabels={0,20,40,60,80,100},ytick={-2,0,2},yticklabels={-2,0,$x^*$}, ymax = 3,ymin=-2.5,xmax=100,xmin=0,xlabel={Time [$s$]},	ylabel={$x$ [$/$]},ylabel shift=-2pt]
		\addplot[smooth,color=blue, fill=blue,line width=0.2mm] table[x index=0, y index=1] {Pictures/Krstic_off_3.dat};
		\addplot[smooth,color=orange, line width=0.2mm, dashed] (0,2)--(100,2);
		\end{axis}
		\end{tikzpicture}
		\label{fig:Krstic_intermittent}}
	\subfigure[Trajectory obtained with intermittent measurements: $T_s=1s$, $\epsilon=0.17s$, and $\omega=20\,\pi\, rad/s$]{
		\begin{tikzpicture}
		\begin{axis}[width=0.4\textwidth,height=4cm,x axis line style={-stealth},y axis line style={-stealth},xtick={0,1,2,3,4},xticklabels={0,1,2,3,4},ytick={-14,-6,2},yticklabels={-14,-6,$x^*$}, ymax = 3,ymin=-14,xmax=4,xmin=0,xlabel={Time [$s$]},ylabel={$x$ [$/$]},ylabel shift=-2pt]
		\coordinate (pt) at (axis cs:0,-2.5);
		\addplot[smooth,color=blue, line width=0.2mm] table[x index=0, y index=1] {Pictures/Krstic_off.dat};
		\addplot[smooth,color=orange, line width=0.2mm, dashed] (0,2)--(4,2);
		\end{axis}
		\node[anchor= north west] at (pt) {
			\begin{tikzpicture}
			\begin{axis}[width=0.15\textwidth,height=2.1cm,x axis line style={-stealth},y axis line style={-stealth},xtick={0,0.1,0.2},xticklabels={\tiny{0},\tiny{0.1},\tiny{0.2}},ytick={-2,-1,0},yticklabels={\tiny{-2},\tiny{-1},\tiny{0}},ymax = 0.5,ymin=-2.5,xmax=0.2,xmin=0,xlabel={\tiny Time [$s$]}, ylabel={\tiny $x$ [/]}, title={\tiny{Zoom}}, title style={yshift=-2.5mm}, xlabel style={yshift=1.5mm},ylabel style={yshift=-1.5mm}]
			\vspace*{-6mm}
			\addplot[smooth,color=blue, line width=0.2mm] table[x index=0, y index=1] {Pictures/Krstic_off_zoom.dat};
			\end{axis}
			\end{tikzpicture}};
	\end{tikzpicture}
	\label{fig:Krstic_intermittent_unstable}}
\subfigure[Trajectory obtained with intermittent measurements: $T_s=1s$, $\epsilon=0.17s$ and $\omega=2002\,\pi\, rad/s$]{
	\begin{tikzpicture}
	\begin{axis}[width=0.4\textwidth,height=4cm,x axis line style={-stealth},y axis line style={-stealth},xtick={0,15,30,45,60},xticklabels={0,15,30,45,60},ytick={-2,0,2},yticklabels={-2,0,$x^*$}, ymax = 3,ymin=-2.5,xmax=60,xmin=0,xlabel={Time [$s$]},	ylabel={$x$ [$/$]}]
	\addplot[smooth,color=blue, fill=blue,line width=0.2mm] table[x index=0, y index=1] {Pictures/Krstic_off_4.dat};
	\addplot[smooth,color=orange, line width=0.2mm, dashed] (0,2)--(60,2);
	\end{axis}
	\end{tikzpicture}
	\label{fig:Krstic_intermittent_ok}}
\caption{Comparison of the trajectories obtained when implementing system \eqref{eq:classical_source} with $h(x)=(x-2)^2+10$, $f_1(h(x)\!)=2\rho h(x)$, $f_2(h(x)\!)=1$, $u_1(t)=\cos(t)$, $u_2(t)=\sin(t)$, $k_1=k_2=1$, $\rho=0.25s^{-1}$, $t_0=0s$, and $x(t_0)=-1$, with continuous (i.e. $\epsilon=T_s$) and intermittent measurements.}
\label{fig:comp_Krstic}
\end{figure}

Those simulation results suggest thus that, for all $\epsilon\in(0,T_s]$, it is still possible to make the extremum seeking system \eqref{eq:classical_source} converge to an (arbitrarily small) neighborhood of the cost minimizer, by selecting a sufficiently large dithers frequency. The following theorem formalizes this suggestion:

\begin{theorem}
\label{lem:classical_source}
Under Assumptions \ref{as:dithers} and \ref{as:cost}-\ref{as:vector_fields_gen}, the point $x^*$ is sGPUAS for system \eqref{eq:classical_source}, with the parameter $\omega^{-1}$.
\end{theorem}
\begin{proof}
The proof is reported in Appendix \ref{ap:proof_classical_source}.\hfill $\blacksquare$
\end{proof}

Interestingly, in both the cases of intermittent and continuous measurements, the point $x^*$ is sGPUAS for the extremum seeking system, with the parameter $\omega^{-1}$. There may be, however, two main drawbacks in feeding the extremum seeking system \eqref{eq:classical_ES} with intermittent measurements (i.e. in implementing \eqref{eq:classical_source}). First, the convergence time of system \eqref{eq:classical_source} is typically larger than the one of system \eqref{eq:classical_ES}. There are indeed $T_s-\epsilon$ seconds that are somehow lost every $T_s$ seconds, since the state follows time-periodic signals with zero-mean and an arbitrary amplitude (i.e. $f_i(0)$) during that period. This may be noticed by comparing Figures \ref{fig:Krstic_intermittent} and \ref{fig:Krstic_intermittent_ok} with Figure \ref{fig:Krstic_continuous}. Second, the nominal dithers frequency to achieve a given steady-state accuracy with system \eqref{eq:classical_source} may be (much) larger than the one needed with system \eqref{eq:classical_ES}. It indeed follows from Figure \ref{fig:Krstic_intermittent_ok} that the nominal dithers frequency to reach a steady-state error smaller than 0.6 is larger than $2002\,\pi$ rad/s, while the result of Figure \ref{fig:Krstic_continuous} (together with other simulations done) suggests that this nominal dithers frequency is less than $20\,\pi$ rad/s for system \eqref{eq:classical_ES}. 

We are now going to adapt system \eqref{eq:classical_source} so as to tackle, or at least reduce, those two drawbacks.

\subsection{A First Modification}
The first change we make on system \eqref{eq:classical_source} ensures that its path is the same as the path of system \eqref{eq:classical_ES} with a time-shift of its initial condition. Accordingly, if a dithers frequency gives satisfying results, in terms of steady-state accuracy and boundedness, with system \eqref{eq:classical_ES}, uniformly in time, then holding it with our modification of system \eqref{eq:classical_source} yields the same results.

The idea is to freeze both the state of system \eqref{eq:classical_source} and the dithers when there is no cost measurement available. This yields the following extremum seeking system
\begin{equation}
\dot{x}(t)\!=\!\left\lbrace\begin{matrix} 0 &\!\!\!\!\!\!\text{if } h_m(t,x(t)\!)\!=\! 0\!\!\!\\
\sum_{i=1}^{l} \sqrt{\omega} f_i(h_m(t,x(t)\!)\!) u_i(k_i \omega \tau(t)\!)  & \text{ else,}
\end{matrix}\right.
\label{eq:modification_ES_intermittent_mod}
\end{equation}
with $\tau(t)=\int_{0}^{t} \sign(h_m(\theta,x(\theta)\!)\!)^2\, d\theta+\tau_0$, and where $\tau_0\in\mathbb{R}$. 

The variable $\tau$ in \eqref{eq:modification_ES_intermittent_mod} represents thus the auxiliary time-variable allowing to halt the dithers when there is no cost measurement. Note that this freezing of the dithers is fundamental to not loose the "averaging" property. 

In terms of practical applications, system \eqref{eq:modification_ES_intermittent_mod} may simply be implemented with $\tau(t)=\int_{t_0}^{t} \sign(h_m(\theta,x(\theta)\!)\!)^2\, d\theta$, i.e. by selecting $\tau_0=-\int_{0}^{t_0} \sign(h_m(\theta,x(\theta)\!)\!)^2\, d\theta$.

Let $t_0\in\mathbb{R}$ and $x_0\in\mathbb{R}^n$. Assuming that $h_m(t,x(t)\!)=0$ if and only if $\modi{t,T_s}\geq \epsilon$, the path of system \eqref{eq:modification_ES_intermittent_mod}, through $x(t_0)=x_0$, is the same as the one of system \eqref{eq:classical_ES}, through $x_c(\tau(t_0)\!)=x_0$. This implies that both the steady-state accuracy and the set in which the trajectory remains bounded are the same for the two systems. The only difference is that the convergence time of system \eqref{eq:modification_ES_intermittent_mod} is longer, since $x$ is only updated $\epsilon$ seconds over $T_s$ seconds. We may thus prove the following stability properties for system \eqref{eq:modification_ES_intermittent_mod}:
\begin{theorem}
\label{lem:stop}
Suppose that $h(x)\neq0$, for all $x\in\mathbb{R}^n$. Then, under Assumptions \ref{as:dithers}, \ref{as:cost}, and \ref{as:vector_fields_gen}, the point $x^*$ is semi-globally practically asymptotically stable for system \eqref{eq:modification_ES_intermittent_mod}, with the parameter $\omega^{-1}$, uniformly in $t_0$ and $\tau_0$.
\end{theorem}
\begin{proof}
The proof is given in Appendix \ref{ap:proof_stop}. \hfill $\blacksquare$
\end{proof}

\begin{remark}
The state of system \eqref{eq:modification_ES_intermittent_mod} is frozen as long as $h_m(t,x(t)\!)=0$. To establish its practical convergence to the cost minimizer, one must thus impose that $h_m(t,x(t)\!)=0$ if and only if there is no cost measurement. Accordingly, since we aim at semi-global results, Theorem \ref{lem:stop} requires that $h(x)\neq 0$, for all $x\in\mathbb{R}^n$. This may seem rather restrictive. However, this assumption is often met in practical applications. In source seeking, for instance, the cost typically corresponds to a scalar field that is strictly positive (e.g. a chemical concentration, a light intensity, a noise intensity, etc.), or that may be made strictly positive by performing a change of scale (e.g. a temperature converted in Kelvin). It is worth to mention that this assumption is also satisfied in the application context of search and rescue of avalanche victims that motivated this work. The cost is indeed the intensity of the electromagnetic field emitted by the avalanche beacon, that is always strictly positive.
\end{remark}

To illustrate the benefit of implementing system \eqref{eq:modification_ES_intermittent_mod}, instead of system \eqref{eq:classical_source}, we consider again the case study that was introduced in Section \ref{sec:basic_scheme}. In Figure \ref{fig:Krstic_wait_1}, one may observe that, for $\omega=20\,\pi$ rad/s and $\epsilon=0.17s$, the trajectory of system \eqref{eq:modification_ES_intermittent_mod} converges to the same set as the one of system \eqref{eq:classical_ES} (cf. Figure \ref{fig:Krstic_continuous}). Namely, the steady-state accuracy is the same as the one obtained with continuous cost measurements. This contrasts with the trajectory of system \eqref{eq:classical_source}, that was diverging (cf. Figure \ref{fig:Krstic_intermittent_unstable}). In Figure \ref{fig:Krstic_wait_2}, one may see that, for $\omega=2002\,\pi$ rad/s and $\epsilon=0.17s$, the steady-state accuracy of system \eqref{eq:modification_ES_intermittent_mod} remains larger than the one of system \eqref{eq:classical_source} (cf. Figure \ref{fig:Krstic_intermittent_ok}). 

\begin{figure}[!ht]
\centering
\subfigure[$\omega=20\,\pi\, rad/s$]{
	\begin{tikzpicture}
	\begin{axis}[width=0.4\textwidth,height=4cm,x axis line style={-stealth},y axis line style={-stealth},xtick={0,20,40,60,80},xticklabels={0,20,40,60,80},ytick={-2,0,2},yticklabels={-2,0,2}, ymax = 3,ymin=-2.5,xmax=80,xmin=0,xlabel={Time [$s$]},	ylabel={$x$ [$/$]}]
	\addplot[smooth,color=blue, fill=blue, line width=0.2mm] table[x index=0, y index=1] {Pictures/Krstic_wait.dat};
	\addplot[smooth,color=orange, line width=0.2mm, dashed] (0,2)--(80,2);
	\end{axis}
	\end{tikzpicture}
	\label{fig:Krstic_wait_1}}
\subfigure[$\omega=2002\,\pi\, rad/s$]{
	\begin{tikzpicture}
	\begin{axis}[width=0.4\textwidth,height=4cm,x axis line style={-stealth},y axis line style={-stealth},xtick={0,20,40,60,80},xticklabels={0,20,40,60,80},ytick={-2,0,2},yticklabels={-2,0,2}, ymax = 3,ymin=-2.5,xmax=80,xmin=0,xlabel={Time [$s$]},	ylabel={$x$ [$/$]}]
	\addplot[smooth,color=blue, fill=blue, line width=0.2mm] table[x index=2, y index=3] {Pictures/Krstic_wait.dat};
	\addplot[smooth,color=orange, line width=0.2mm, dashed] (0,2)--(80,2);
	\end{axis}
	\end{tikzpicture}
	\label{fig:Krstic_wait_2}}
\caption{Trajectories obtained when implementing system \eqref{eq:modification_ES_intermittent_mod} with $h(x)=(x-2)^2+10$, $f_1(h(x)\!)=2\rho h(x)$, $f_2(h(x)\!)=1$, $u_1(t)=\cos(t)$, $u_2(t)=\sin(t)$, $k_1=k_2=1$, $\rho=0.25s^{-1}$, $t_0=0s$, $x(t_0)=-1$, $T_s=1s$, and $\epsilon=0.17s$.}
\label{fig:Krstic_wait}
\end{figure}
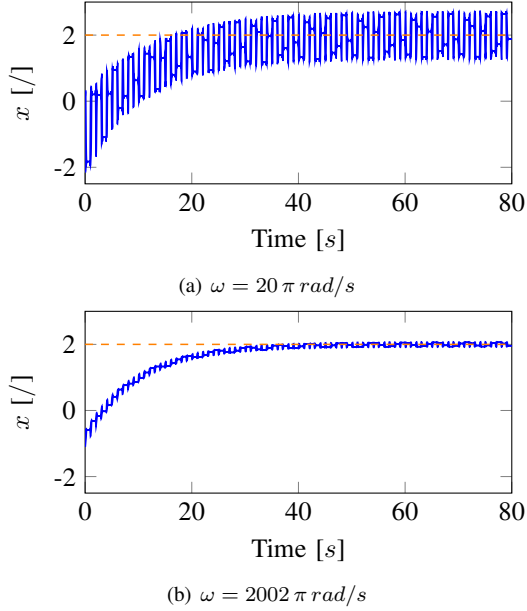

There is still, however, a main drawback to system \eqref{eq:modification_ES_intermittent_mod}. Its convergence time remains (much) larger than the one of system \eqref{eq:classical_ES} (compare Figure \ref{fig:Krstic_wait_1} with Figure \ref{fig:Krstic_continuous}). This comes from the fact that the state is not updated in absence of cost measurements. There are thus $T_s-\epsilon$ seconds lost every $T_s$ seconds. In the next section, we propose a way to avoid this issue.
\subsection{Using a Previous Estimate of the Gradient}
To reduce the convergence time of system \eqref{eq:modification_ES_intermittent_mod}, we have to find a way to steer its state towards the cost optimizer when there is no cost measurement. The idea is to take benefit from the previous cost measurements to get an estimation of the gradient. The state might then be updated along that direction until the cost may be measured again. 

We will now see a possible way to obtain such an estimation of the gradient. We will then show that this estimation may also be used to adapt the dithers amplitude, so as to improve the steady-state accuracy.
\subsubsection{Non-adaptive Dithers Amplitude}
We know that the trajectory of system \eqref{eq:modification_ES_intermittent_mod} approximates the one of the gradient-descent law, on each time-interval where the cost is measured. By considering the average value of its dynamics, on a dithers period, one should thus be able to get an estimation of the gradient. The following lemma formalizes this idea:
\begin{lemma}
	\label{lem:approx_gradient_average}
	Suppose that Assumptions \ref{as:dithers}, \ref{as:cost}, and \ref{as:vector_fields_gen}, hold. Let $C:=2\pi \textbf{LCM}(k_1^{-1},k_2^{-1},...,k_l^{-1})$ and $\epsilon'\in(0,\epsilon]$. Furthermore, let the compact set $\mathcal{W}\subset\mathbb{R}^n$, and the constants  $\omega\in[C/\epsilon',\infty)$, and $k_0\in\mathbb{Z}$, be such that the trajectory of system \eqref{eq:modification_ES_intermittent_mod} satisfies $x(t)\in\mathcal{W}$, for all $t\in[k_0T_s+\epsilon'-T,k_0T_s+\epsilon']$, with $T:=C\omega^{-1}$. Defining
	\begin{equation}
		g(t)=\sqrt{\omega}/T\sum_{i=1}^{l}\int_{t-T}^{t}  f_i(h_m(\theta,x(\theta)\!)\!) u_i(k_i\omega \theta)\, d\theta,
		\label{eq:def_g}
	\end{equation}
	it holds then
	\begin{equation}
	\norm{g\left(k_0T_s+\epsilon'\right)+\rho \nabla h\left(x\left(k_0T_s+\epsilon'\right)\!\right)}\\
	\!<\!l^3MC^{2}\omega^{-0.5},
	\label{eq:bound_approx_gradient_lemma}
	\end{equation}
	with $M\in\mathbb{R}_{>0}$ such that  $\norm{\pazocal{L}_{f_m}\pazocal{L}_{f_j}f_i(h(x)\!)}<M$, for all $x\in\mathcal{W}$ and $i,j,m\in\{1,2,...,l\}$.
\end{lemma}
\begin{proof}
	The proof can be found in Appendix \ref{ap:proof_gradient_average}. \hfill $\blacksquare$
\end{proof}

As stated in Lemma \ref{lem:approx_gradient_average}, $g(k_0T_s+\epsilon)$ provides an estimation of $-\rho\nabla h(k_0T_s+\epsilon)$, with an accuracy that may be made as large as desired, by increasing $\omega$. We can thus make the extremum seeking system \eqref{eq:modification_ES_intermittent_mod} approximate a "discrete-time" gradient-descent law when there is no cost measurement. This yields

\begin{equation}
 \dot{x}(t)=\left\lbrace\begin{matrix}\sum_{i=1}^{l}\sqrt{\omega} f_i(h_m(t,x(t)\!)\!) u_i(k_i\omega t)& \\ &\!\!\!\!\!\!\!\!\!\!\!\!\!\!\!\!\!\!\!\!\!\!\!\!\!\!\!\!\!\!\!\!\!\!\!\!\!\!\!\!\!\!\!\!\!\!\!\!\text{ if } \modi{t,T_s}<\floor{\epsilon'/T}T\\
 \rho_2 g\left(\floor{t/T_s}T_s+\floor{\epsilon'/T}T\right) &\\ &\!\!\!\!\!\!\!\!\!\quad\text{ else }
 \end{matrix}\right.,
 \label{eq:solution_ON_mod}
\end{equation}
with $\rho_2\in\mathbb{R}_{>0}$, $\epsilon'\in(0,\epsilon]$, and $g(t)$ defined in \eqref{eq:def_g}.

\begin{remark}
	Note that system \eqref{eq:solution_ON_mod} only follows the dynamics of \eqref{eq:classical_source} on the time-interval $[kT_s,kT_s+\floor{\epsilon'/T}T]$, while system \eqref{eq:modification_ES_intermittent_mod} follows it on the time-interval $[kT_s,kT_s+\epsilon]$. We indeed saw in Figure \ref{fig:Krstic_intermittent_unstable} that having an incomplete dithers period may destroy (part of) what has been done during the complete dithers periods. Since $\epsilon$ is  unknown, we assumed that an estimation $\epsilon'\in(0,\epsilon]$ of $\epsilon$ was available, so as to be able to work on complete dithers period. Having such an estimation at hand seems reasonable. It is for instance the case in the search and rescue application mentioned in the introduction. If this estimation is, however, not available, we know from the proof of Theorem \ref{lem:classical_source} that, by selecting a sufficiently large dithers frequency, it is still possible to make the effect of the incomplete dithers period as small as desired. It would thus be possible to make system \eqref{eq:solution_ON_mod} follow the dynamics of \eqref{eq:classical_source} on the time-interval $[kT_s,kT_s+\epsilon]$. However, this would be less efficient.
\end{remark}
 
For $\rho_2$ and $\omega^{-1}$ sufficiently small, $x$ is thus expected to converge in a neighborhood of $x^*$. This intuition is formalized in the following theorem:
\begin{theorem}
	\label{lem:stability_gradient_average}
	Under Assumptions \ref{as:dithers}, \ref{as:cost} and \ref{as:vector_fields_gen}, the point $x^*$ is sGPUAS for system \eqref{eq:solution_ON_mod}, with the vector of parameters $[\omega^{-1},\rho_2]$.
\end{theorem}

\begin{proof}
	The proof can be found in Appendix \ref{ap:proof_stability_gradient_average}. \hfill $\blacksquare$
\end{proof}

\begin{remark}
	\label{rem:implementation}
A few comments have to be done regarding the practical implementation of system \eqref{eq:solution_ON_mod}. First, since $T_s$ is unknown, the switching condition cannot be implemented as such. One would thus proceed in the equivalent way. As soon as the cost measurement becomes available (i.e. when $h_m(t,x(t)\!)$ switches from a zero value to a non-zero one), one implements the first case of \eqref{eq:solution_ON_mod}. After $\floor{\epsilon'/T}T$ seconds, one switches to the second case. The initialization of the algorithm also requires some care. It is indeed better to wait until $h_m(t,x(t)\!)$ switches from a zero value to a non-zero one to start the algorithm. Indeed, even if $h_m(t,x(t)\!)\neq 0$ when starting the algorithm, there is no way to know whether there will be at least a dithers period before the end of the measurement. One has thus not the guarantee of being able to get an estimation of the gradient when the measurement stops.  
\end{remark}

To illustrate the benefits of the proposed modification, we consider again our case study, introduced in Section \ref{sec:basic_scheme}, with $\epsilon'=0.1s$. In Figure \ref{fig:Krstic_grad_1}, one may observe that the convergence time obtained with system \eqref{eq:solution_ON_mod} is about eight times smaller than the one obtained with system \eqref{eq:modification_ES_intermittent_mod} (cf. Figure \ref{fig:Krstic_wait_1}), for the same steady-state error. Considering again Figure \ref{fig:Krstic_continuous}, one may see that this convergence time is similar to the one obtained with continuous measurements. It is also worth to mention that system \eqref{eq:solution_ON_mod} may also be used with other vector fields. One may for instance select $f_1(h(x)\!)=\sqrt{2\rho}\cos(h(x)\!)$ and $f_2(h(x)\!)=-\sqrt{2\rho} \sin(h(x)\!)$, as proposed in \cite{Sch14} for continuous measurements. This leads to the results presented in Figure \ref{fig:Scheinker_grad}. It may be observed that, for given values of $\omega$, $\rho$, and $\rho_2$, the amplitude of the oscillations are much smaller, yielding to a smaller steady-state error. One may also see in Figure \ref{fig:Scheinker_grad} that increasing $\rho_2$ reduces the convergence time. Note, however, that above a given threshold, increasing further $\rho_2$ may lead to overshoots and destabilize the system.

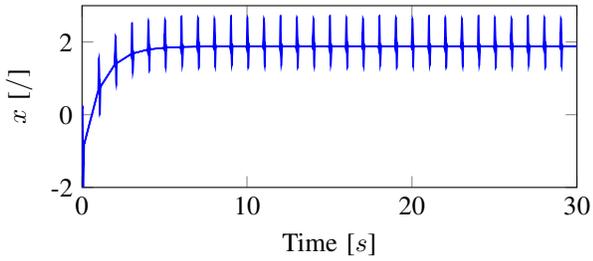
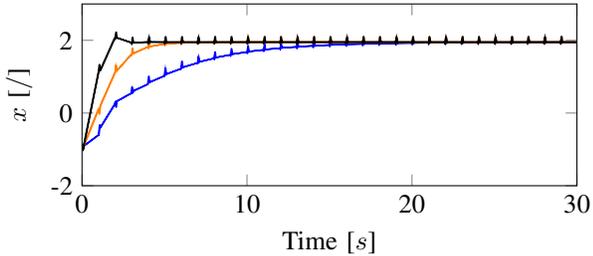
\begin{figure}[!ht]
\centering
\subfigure[$f_1(h(x)\!)=2\rho h(x)$, $f_2(h(x)\!)=1$ and $\rho_2=1.5\rho$ (\full{blue})]{
	\begin{tikzpicture}
	\begin{axis}[width=0.45\textwidth,height=4cm,x axis line style={-stealth},y axis line style={-stealth},xtick={0,10,20,30},xticklabels={0,10,20,30},ytick={-2,0,2},yticklabels={-2,0,2}, ymax = 3,ymin=-2,xmax=30,xmin=0,xlabel={Time [$s$]},	ylabel={$x$ [$/$]}]
	\addplot[smooth,color=blue, fill=blue, line width=0.2mm] table[x index=2, y index=3] {Pictures/Krstic_grad.dat};
	\end{axis}
	\end{tikzpicture}
	\label{fig:Krstic_grad_1}}
\hfill
\subfigure[$f_1(h(x)\!)=\sqrt{2\rho} \cos(h(x)\!)$ and $f_2(h(x)\!)=-\sqrt{2\rho} \sin(h(x)\!)$: $\rho_2=1.5\rho$ (\full{blue}), $\rho_2=5\rho$ (\full{orange}), and $\rho_2=10\rho$ (\full{black})]{
	\begin{tikzpicture}
	\begin{axis}[width=0.45\textwidth,height=4cm,x axis line style={-stealth},y axis line style={-stealth},xtick={0,10,20,30},xticklabels={0,10,20,30},ytick={-2,0,2},yticklabels={-2,0,2}, ymax = 3,ymin=-2,xmax=30,xmin=0,xlabel={Time [$s$]},	ylabel={$x$ [$/$]}]
	\addplot[smooth,color=blue, fill=blue, line width=0.2mm] table[x index=0, y index=1] {Pictures/Scheinker_grad.dat};
	\addplot[smooth,color=orange, fill=orange, line width=0.2mm] table[x index=2, y index=3] {Pictures/Scheinker_grad.dat};
	\addplot[smooth,color=black, fill=black, line width=0.2mm] table[x index=4, y index=5] {Pictures/Scheinker_grad.dat};
	\end{axis}
	\end{tikzpicture}
	\label{fig:Scheinker_grad}}
\caption{Comparison of the trajectories obtained when implementing system \eqref{eq:solution_ON_mod} for two choices of vector fields, and with $h(x)=(x-2)^2+10$, $\rho_2=1.5\rho$, $u_1(t)=\cos(t)$, $u_2(t)=\sin(t)$, $k_1=k_2=1$, $\rho=0.25s^{-1}$, $t_0=0s$, $x(t_0)=-1$, $T_s=1s$, $\epsilon=0.17s$, and $\omega=20\,\pi\, rad/s$.}
\label{fig:comparison_grad}
\end{figure}
\subsubsection{Adaptive Dithers Amplitude}
We will now take benefit from the estimation of $\nabla h(kT_s+\floor{\epsilon'/T}T)$ to adapt the dithers amplitude, so as to improve the steady-state accuracy. Note that this adaptation may also be used for extremum seeking systems with continuous measurements.

The idea is to make the dithers amplitude vanish as the norm of the estimated gradient tends to zero. The proposed modification of system \eqref{eq:solution_ON_mod} is
\begin{equation}
\dot{x}(t)\!=\!\left\lbrace\begin{matrix} \sqrt{\omega}\alpha(t)\sum_{i=1}^{l} f_i(h_m(t,x(t)\!)\!) u_i(k_i\omega t) &\\ &\!\!\!\!\!\!\!\!\!\!\!\!\!\!\!\!\!\!\!\!\!\!\!\!\!\!\!\!\!\!\!\!\!\!\!\!\!\!\!\!\!\!\!\!\!\!\!\!\!\!\!\!\!\!\!\!\!\!\!\!\!\!\!\!  \text{if } \modi{t,T_s}<\floor{\epsilon'/T}T\\
\rho_2 g(\floor{t/T_s}T_s+\floor{\epsilon'/T}T) \qquad  \text{else,}&
\end{matrix}\right.\!\!\!
\label{eq:solution_ON_VDA_1}
\end{equation}
with 
\begin{equation}	\alpha(t)\!=\!\sqrt{\frac{\norm{g\left(\!(\floor{t/T_s}-1)T_s+\floor{\epsilon'/T}T\right)}+a}{\norm{g\left(\!(\floor{t/T_s}-1)T_s+\floor{\epsilon'/T}T\right)}+b}}
\end{equation}
and
\begin{multline}
g(t)=\frac{\sqrt{\omega}}{T}\sum_{i=1}^{l}\int_{t-T}^{t} \sqrt{ \frac{\norm{g(t-T_s)}+b}{\norm{g(t-T_s)}+a}}\times\\ f_i(h_m(\tau,x(\tau)\!)\!) u_i(k_i\omega \tau) d\tau,
\label{eq:deg_g_VDA}
\end{multline}
with $a\in\mathbb{R}_{>0}$, $b\in(a,\infty)$, and where the dithers and vector fields satisfy Assumptions \ref{as:dithers} and \ref{as:vector_fields_gen}. 

The parameter $b$ is introduced to ensure that $\alpha(t)$ exists, for all $t\in\mathbb{R}$, and that its value is always bounded by 1. The value of $b$ determines the threshold above which the dithers amplitude starts to "proportionally" decrease with $\norm{g}$. The parameter $a$ is an arbitrarily small parameter introduced to ensure that the inverse of $\alpha(t)$, used to compute $g(t)$, exists, for all $t\in\mathbb{R}$. Since $\floor{t/T_s}$ is constant for $t\in[kT_s,kT_s+\epsilon']$, the trajectory of system \eqref{eq:solution_ON_VDA_1} still approximates the trajectory of the gradient-descent law on that time-interval. The only difference is that the update rate is adapted from one period to the other. Furthermore, it may be shown that $g(\floor{t/T_s}T_s+\floor{\epsilon'/T}T)$ still provides an estimation of $-\rho \nabla h(x(\floor{t/T_s}T_s+\floor{\epsilon'/T}T)\!)$, as in \eqref{eq:solution_ON_mod}. System \eqref{eq:solution_ON_VDA_1} is thus expected to inherit from the properties of system \eqref{eq:solution_ON_mod}. For space limits, the rigorous characterization of the stability properties is, however, omitted.

The practical implementation of system \eqref{eq:solution_ON_VDA_1} is similar to the one of \eqref{eq:solution_ON_mod} (see Remark \ref{rem:implementation}). Note that, for the first period, the value of $\norm{g\left(\!(\floor{t/T_s}-1)T_s+\floor{\epsilon'/T}T\right)}$ is selected arbitrarily. 

To see the advantages of using variable dithers amplitude, we consider again our case study (see Section \ref{sec:basic_scheme}). Comparing Figure \ref{fig:grad_VDA} with Figure \ref{fig:comparison_grad}, one may see that the amplitude of the oscillations is smaller, for the same convergence time. Furthermore, for the vector fields selected as in \cite{Sch14}, one may see in Figure \ref{fig:grad_VDA_Sch} that, after 10s, the oscillations become invisible. By zooming on the steady-state, the error amounts now to $10^{-3}$, which is about 100 times smaller than without the adaptation of the dithers amplitude (cf. Figure \ref{fig:Scheinker_grad}). 
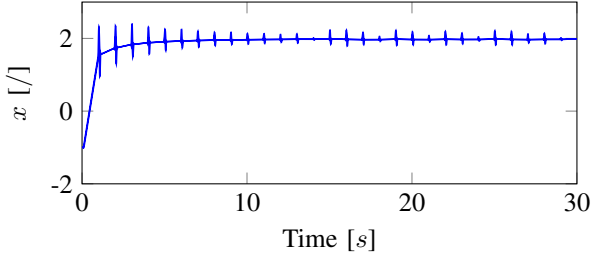
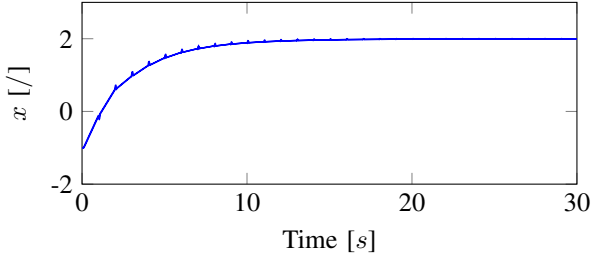
\begin{figure}[!ht]
\centering
\subfigure[$f_1(h(x)\!)=2\rho h(x)$ and $f_2(h(x)\!)=1$]{
	\begin{tikzpicture}
	\begin{axis}[width=0.45\textwidth,height=4cm,x axis line style={-stealth},y axis line style={-stealth},xtick={0,10,20,30},xticklabels={0,10,20,30},ytick={-2,0,2},yticklabels={-2,0,2}, ymax = 3,ymin=-2,xmax=30,xmin=0,xlabel={Time [$s$]},	ylabel={$x$ [$/$]}]
	\addplot[smooth,color=blue, fill=blue, line width=0.2mm] table[x index=0, y index=1] {Pictures/grad_VDA_Krs.dat};
	\end{axis}
	\end{tikzpicture}
	\label{fig:grad_VDA_Krs}}
\hfill
\subfigure[$f_1(h(x)\!)=\sqrt{2\rho} \cos(h(x)\!)$ and $f_2(h(x)\!)=-\sqrt{2\rho} \sin(h(x)\!)$]{
	\begin{tikzpicture}
	\begin{axis}[width=0.45\textwidth,height=4cm,x axis line style={-stealth},y axis line style={-stealth},xtick={0,10,20,30},xticklabels={0,10,20,30},ytick={-2,0,2},yticklabels={-2,0,2}, ymax = 3,ymin=-2,xmax=30,xmin=0,xlabel={Time [$s$]},	ylabel={$x$ [$/$]}]
	\addplot[smooth,color=blue, fill=blue, line width=0.2mm] table[x index=0, y index=1] {Pictures/grad_VDA.dat};
	\end{axis}
	\end{tikzpicture}
	\label{fig:grad_VDA_Sch}}
\caption{Comparison of the trajectories obtained when implementing system \eqref{eq:solution_ON_VDA_1} for two choices of vector fields, and with $h(x)=(x-2)^2+10$, $\rho_2=1.5\rho$, $u_1(t)=\cos(t)$, $u_2(t)=\sin(t)$, $k_1=k_2=1$, $\rho=0.25s^{-1}$, $t_0=0s$, $x(t_0)=-1$, $T_s=1s$, $\epsilon=0.17s$, $\rho_2=5\rho$, $a=10^{-5}$, $b=0.1$, and $\omega=20\,\pi\, rad/s$.}
\label{fig:grad_VDA}
\end{figure}
\section{Conclusion}
In this work, we considered the case of intermittent cost measurements in the framework of extremum seeking. We first showed the ability of a class of existing extremum seeking systems to steer the cost input towards the minimizer, despite the absence of cost measurement on some time-intervals. However, compared with the case of continuous measurements, this may result in a lower convergence rate and may require the use of a larger dithers frequency. We then proposed three modifications allowing to reduce the convergence time, while keeping a reasonable dithers frequency. The performances of the different schemes were compared on a case study. Future works may consist in considering a more general set-up. For instance, a nonlinear dynamical system might be associated with the cost function, presence of noise on the cost measurement might be taken into account, and higher-order Lie-bracket based extremum seeking systems might be considered. One may also think of applying the proposed extremum seeking systems on a more physical case study, such as the search and rescue of victims of avalanches with an ARVA beacon.
\appendices
\section{Proof of Theorem \ref{lem:classical_source}}
\label{ap:proof_classical_source}
For the sake of compactness, let us introduce $C\!=\!2\pi \textbf{LCM}(k_1^{-1},k_2^{-1},...,k_l^{-1})$ and $T\!=\!C\omega^{-1}$. Furthermore, let $\delta_V\in\mathbb{R}_{>0}$, and $\delta_B\in\mathbb{R}_{>0}$, be arbitrary, but fixed. Select then $\delta_V'\in(\delta_V,\infty)$, $\delta_W=2\delta_V'-\delta_V$, $\delta_Q'\in(0,\mins{\delta_B,\delta_V})$, and $\delta_Q\in(0,\delta_Q')$. 

To perform the proof, we proceed in two main steps. In \textit{Step 1}, we analyze the evolution of the trajectory of system \eqref{eq:classical_source} on the interval $[kT_s,(k+1)T_s)$, for $k\in\mathbb{Z}$. In \textit{Step 2}, we exploit the results of \textit{Step 1} to show that the three conditions of Definition \ref{def:sgpuas} are enforced. 

\textit{Step 1.} Given the composite nature of system \eqref{eq:classical_source}, we successively examine its trajectory on the two sub-intervals $[kT_s,kT_s+\epsilon)$ and $[kT_s+\epsilon,(k+1)T_s)$, for $k\in\mathbb{Z}$.

\textit{Step 1.1.} We know from Proposition 1 and Lemma 2 in \cite{Lab22} that there exist an $\omega^*_1\in\left(C\epsilon^{-1},\infty\right)$, an $M_V\in\mathbb{R}_{>0}$, and a $B\in\mathbb{R}_{>0}$, such that, for every $\omega\in(\omega^*_1,\infty)$, $k\in\mathbb{Z}$, $t_0\in[kT_s,kT_s+\epsilon)$, and $x_0\in U_{\mathbb{R}^n}^{x^*}(\delta_W)$, the trajectory of system \eqref{eq:classical_source}, through $x(t_0)=x_0$, satisfies
\begin{multline}
\!\!\!\norm{x(t)-x(t_0)}\leq \min\!\left\lbrace B\omega^{-0.5},0.5\!\mins{\delta_B,\delta_V}-0.5\delta_Q',\right.\\\left.0.5(\delta_V'-\delta_V),0.5\delta_Q,0.5(\delta_Q'-\delta_Q)\right\rbrace,
\label{eq:traj_on_std}
\end{multline}  
for all $t\in[t_0,t_0+\min\left\lbrace\epsilon-\modi{t_0,T_s},T\right\rbrace]$ and, if $\epsilon-\modi{t_0,T_s}\geq T$, 
\begin{multline}
\norm{x(t_0+T)-x^*}^2\leq \norm{x_0-x^*}^2\\-2T\left[\rho\nabla^T h(x_0)(x_0-x^*)-M_V\omega^{-0.5}\right].
\label{eq:decrease_norm}
\end{multline}
Given Assumption \ref{as:cost}, it holds $\nabla^T h(x) (x-x^*)>0$, for all $x\in\mathbb{R}^n\backslash\{x^*\}$. Let $m_g:=\inf_{x\in U_{\mathbb{R}}^{x^*}(\delta_W)\backslash U_{\mathbb{R}}^{x^*}(0.5\delta_Q)} \rho \nabla^T h(x)(x-x^*)$. Selecting $\omega^*_2:=\maxs{\omega_1^*, 4 M_V^2/m_g^2}$, it results from \eqref{eq:decrease_norm} that, for all $\omega\in(\omega^*_2,\infty)$, $k\in\mathbb{Z}$, $t_0\in[kT_s,kT_s+\epsilon-T)$, and $x_0\in U_{\mathbb{R}^n}^{x^*}(\delta_W)\backslash U_{\mathbb{R}^n}^{x^*}(0.5\delta_Q)$, the trajectory of system \eqref{eq:classical_source}, through $x(t_0)=x_0$, satisfies
\begin{equation}
\norm{x(t_0+T)-x^*}^2<\norm{x_0-x^*}^2-m_gT.
\label{eq:lyap_std}
\end{equation}

\textit{Step 1.2.} For all $k\in\mathbb{Z}$, $t_0\in[kT_s+\epsilon,(k+1)T_s)$, $\omega\in\mathbb{R}_{>0}$, and $x_0\in\mathbb{R}^n$, the trajectory of system \eqref{eq:classical_source}, through $x(t_0)=x_0$, satisfies
\begin{equation}
x(t)=x_0+\sqrt{\omega}\sum_{i=1}^{l} f_i(0) \int_{t_0}^{t} u_i(k_i\omega \tau)\, d\tau,
\label{eq:trajectory_OFF_classical}
\end{equation}
for all $t\in[t_0,(k+1)T_s]$. In virtue of Assumption \ref{as:dithers}, for all $i\in \{1,2,...,l\}$, it holds $\vert u_i(t)\vert\leq 1$,   for all $t\in\mathbb{R}$, and $\int_{t_0}^{t_0+2\pi} u_i(\tau)\, d\tau=0$, for all $t_0\in\mathbb{R}$. It follows thus from \eqref{eq:trajectory_OFF_classical} that, for all $k\in\mathbb{Z}$, $t_0\in[kT_s+\epsilon,(k+1)T_s)$, $\omega\in\mathbb{R}_{>0}$, and $x_0\in\mathbb{R}^n$, the trajectory of system \eqref{eq:classical_source}, through $x(t_0)=x_0$, satisfies
\begin{equation}
\norm{x(t)-x_0}\leq M_S \omega^{-0.5},
\label{eq:bound_traj_std}
\end{equation}
for all $t\in[t_0,(k+1)T_s]$, with $M_S:=\sum_{i=1}^{l} \norm{f_i(0)} C$. Let $\omega^*_3=\max\left\lbrace\omega_2^*,4M_S^2/\min\left\lbrace(\delta_V'-\delta_V)^2,(\delta_Q'-\delta_Q)^2\right\rbrace\!\right\rbrace$.

\textit{Step 2.} Let us now prove that the three properties of Definition \ref{def:sgpuas} are fulfilled. As introduced in Section \ref{sec:basic_scheme}, the idea is the following. For sufficiently large dithers frequencies, what occurs during the complete dithers periods, when the cost can be measured, is not too impacted by what happens on the rest of the transmission period. Since the initial time (i.e. $t_0$ in Definition \ref{def:sgpuas}) is arbitrary, the first transmission may be incomplete, preventing the use of this argument. In the sequel, we consider thus the first transmission period separately. We then use recursive arguments to handle the next periods. 

\textit{Step 2.1: Boundedness property.} It follows from \textit{Step 1.1.} (see \eqref{eq:traj_on_std}) that, for all $\omega\in(\omega^*_3,\infty)$, $k\in\mathbb{Z}$, and $t_0\in[kT_s,kT_s+\epsilon)$, if $x(t_0)\in U_{\mathbb{R}^n}^{x^*}(\delta_V)$, then $x(t)\in U_{\mathbb{R}^n}^{x^*}(0.5\delta_V+0.5\delta_V')$, for all $t\in[t_0,t_0+\mins{\epsilon-\modi{t_0,T_s},T}]$. In addition, for all $\omega\in(\omega^*_3,\infty)$, $k\in\mathbb{Z}$, $t_0\in[kT_s,kT_s+\epsilon-T)$, and $x(t_0)\in U_{\mathbb{R}^n}^{x^*}(\delta_V)\backslash U_{\mathbb{R}^n}^{x^*}(0.5\delta_Q)$, it holds (see \eqref{eq:lyap_std}) $\norm{x(t_0+T)-x^*}<\norm{x(t_0)-x^*}$ and, hence, $x(t_0+T)\in U_{\mathbb{R}^n}^{x^*}(\delta_V)$. Furthermore, for all $\omega\in(\omega^*_3,\infty)$, $k\in\mathbb{Z}$, $t_0\in[kT_s,kT_s+\epsilon-T)$, and $x(t_0)\in U_{\mathbb{R}^n}^{x^*}(0.5\delta_Q)$, one has (see \eqref{eq:traj_on_std}) $x(t_0+T)\in U_{\mathbb{R}^n}^{x^*}(\delta_Q)$ and, hence, $x(t_0+T)\in U_{\mathbb{R}^n}^{x^*}(\delta_V)$. Combining those results allows us to conclude that, for all $\omega\in(\omega^*_3,\infty)$, $k\in\mathbb{Z}$, and $t_0\in[kT_s,kT_s+\epsilon)$, if $x(t_0)\in U_{\mathbb{R}^n}^{x^*}(\delta_V)$, then $x(t)\in U_{\mathbb{R}^n}^{x^*}(0.5\delta_V+0.5\delta_V')$, for all $t\in[t_0,kT_s+\epsilon]$. Furthermore, we know from \textit{Step 1.2.} (see \eqref{eq:bound_traj_std}) that, for all $\omega\in(\omega^*_3,\infty)$, $k\in\mathbb{Z}$, $t_0\in[kT_s+\epsilon,(k+1)T_s)$, and $x(t_0)\in\mathbb{R}^n$, it holds $\norm{x(t)-x(t_0)}\leq M_S \omega^{-0.5}$, for all $t\in[t_0,(k+1)T_s]$. We may thus conclude that, for all $\omega\in(\omega^*_3,\infty)$, $t_0\in\mathbb{R}$, and $x_0\in U_{\mathbb{R}^n}^{x^*}(\delta_V)$, the trajectory of system \eqref{eq:classical_source}, through $x(t_0)=x_0$, satisfies $x(t)\in U_{\mathbb{R}^n}^{x^*}(\delta_V')$ and, hence , $x(t)\in U_{\mathbb{R}^n}^{x^*}(\delta_W)$, for all $t\in[t_0,t_0+T_s-\modi{t_0,T_s}]$. To conclude the boundedness property, we will now show the existence of an $\omega_4^*\in(\omega^*_3,\infty)$ such that, for all $\omega\in(\omega^*_4,\infty)$, $k\in\mathbb{Z}$, and $x_0\in U_{\mathbb{R}^n}^{x^*}(\delta_V')$, the trajectory of system \eqref{eq:classical_source}, through $x(kT_s)=x_0$, satisfies $x(t)\in U_{\mathbb{R}^n}^{x^*}(\delta_W)$, for all $t\in[kT_s,(k+1)T_s]$, and $x((k+1)T_s)\in U_{\mathbb{R}^n}^{x^*}(\delta_V')$.
It may already be concluded from the choice of $\omega^*_3$ that, if $x(kT_s)\in U_{\mathbb{R}^n}^{x^*}(\delta_V')$, then (see \eqref{eq:traj_on_std} and \eqref{eq:lyap_std}) $x(t)\in U_{\mathbb{R}^n}^{x^*}(\delta_W)$, for all $t\in[kT_s,kT_s+\floor{\epsilon/T}T]$, and $x(kT_s+\floor{\epsilon/T}T)\in U_{\mathbb{R}^n}^{x^*}(\delta_V')$. Two cases may then be distinguished: either $x(kT_s+\floor{\epsilon/T}T)\in U_{\mathbb{R}^n}^{x^*}(\delta_V)$, and we already know from the above result that $x(t)\in U_{\mathbb{R}^n}^{x^*}(\delta_V')$ and, hence, $x(t)\in U_{\mathbb{R}^n}^{x^*}(\delta_W)$, for all $t\in[kT_s+\floor{\epsilon/T}T, (k+1)T_s]$, or $x(kT_s+\floor{\epsilon/T}T)\in U_{\mathbb{R}^n}^{x^*}(\delta_V')\backslash U_{\mathbb{R}^n}^{x^*}(\delta_V)$. The latter case implies that $x(kT_s+lT)\in U_{\mathbb{R}^n}^{x^*}(\delta_V')\backslash U_{\mathbb{R}^n}^{x^*}(\delta_V)$, for all $l\in\{0,...,(\floor{\epsilon/T}-1)T\}$. Therefore, it results from the choice of $\omega_3^*$ that (see \eqref{eq:lyap_std}) $\norm{x(kT_s+\floor{\epsilon/T}T)-x^*}^2<\norm{x(kT_s)-x^*}^2-m_g(\epsilon-C/\omega_3^*)$. Selecting $\omega_4^*:=\maxs{\omega_3^*,(M_S+B)^2/(\delta_V'-\sqrt{\delta_V'^2-m_g(\epsilon-C/\omega_3^*)})^2}$, it follows then that $\norm{x(t)-x^*}< \sqrt{\delta_V'^2-m_g(\epsilon-C/\omega_3^*)}+(B+M_S)/\sqrt{\omega}<\delta_V'$ and, hence, $x(t)\in U_{\mathbb{R}^n}^{x^*}(\delta_W)$, for all $\omega\in(\omega^*_4,\infty)$, and $t\in[kT_s+\floor{\epsilon/T}T,(k+1)T_s]$, concluding the boundedness property.

\textit{Step 2.2: Stability property.} Let us follow a reasoning similar to the one used in \textit{Step 2.1}. The results of \textit{Step 1.1.} (see \eqref{eq:traj_on_std}) ensure that, for all $\omega\in(\omega_4^*,\infty)$, $k\in\mathbb{Z}$, $t_0\in[kT_s,kT_s+\epsilon)$, and $x(t_0)\in U_{\mathbb{R}^n}^{x^*}(\delta_Q)$, one has $x(t)\in U_{\mathbb{R}^n}^{x^*}(0.5\delta_Q+0.5\delta_Q')$, for all $t\in[t_0,t_0+\mins{\epsilon-\modi{t_0,T_s},T}]$. Moreover, for all $\omega\in(\omega_4^*,\infty)$, $k\in\mathbb{Z}$, $t_0\in[kT_s,kT_s+\epsilon-T)$, and $x(t_0)\in U_{\mathbb{R}^n}^{x^*}(\delta_Q)\backslash U_{\mathbb{R}^n}^{x^*}(0.5\delta_Q)$, it holds (see \eqref{eq:lyap_std}) $\norm{x(t_0+T)-x^*}<\norm{x(t_0)-x^*}$ and, hence, $x(t_0+T)\in U_{\mathbb{R}^n}^{x^*}(\delta_Q)$. In addition, for all $\omega\in(\omega_4^*,\infty)$, $k\in\mathbb{Z}$, $t_0\in[kT_s,kT_s+\epsilon-T)$, and $x(t_0)\in U_{\mathbb{R}^n}^{x^*}(0.5\delta_Q)$, it holds (see \eqref{eq:traj_on_std}) $x(t_0+T)\in U_{\mathbb{R}^n}^{x^*}(\delta_Q)$. Furthermore, we know from \textit{Step 1.2.} (see \eqref{eq:bound_traj_std}) that, for all $\omega\in(\omega_4^*,\infty)$, $k\in\mathbb{Z}$, $t_0\in[kT_s+\epsilon,(k+1)T_s)$, and $x(t_0)\in \mathbb{R}^n$, one has $\norm{x(t)-x(t_0)}<0.5(\delta_Q'-\delta_Q)$, for all $t\in[t_0,(k+1)T_s]$. We may thus already conclude that, for all $\omega\in(\omega^*_4,\infty)$, $t_0\in\mathbb{R}$, and $x(t_0)\in U_{\mathbb{R}^n}^{x^*}(\delta_Q)$, it holds $x(t)\in U_{\mathbb{R}^n}^{x^*}(\delta_Q')$, for all $t\in[t_0,t_0+T_s-\modi{t_0,T_s}]$. To conclude the stability property, we will now show the existence of an $\omega^*_5\in(\omega^*_4,\infty)$ such that, for every $\omega\in(\omega^*_5,\infty)$ and $k\in\mathbb{Z}$, if $x(kT_s)\in U_{\mathbb{R}^n}^{x^*}(\delta_Q')$, then $x(t)\in U_{\mathbb{R}^n}^{x^*}(\delta_B)$, for all $t\in[kT_s,(k+1)T_s]$, and $x((k+1)T_s)\in U_{\mathbb{R}^n}^{x^*}(\delta_Q')$. To do so, first note that the choice of $\omega_4^*$ ensures that, for all $\omega\in(\omega^*_4,\infty)$, if $x(kT_s)\in U_{\mathbb{R}^n}^{x^*}(\delta_Q')$, then (see \eqref{eq:traj_on_std} and \eqref{eq:lyap_std}) $x(t)\in U_{\mathbb{R}^n}^{x^*}(\delta_B)$, for all $t\in[kT_s,kT_s+\floor{\epsilon/T}T]$, and $x(kT_s+\floor{\epsilon/T}T)\in U_{\mathbb{R}^n}^{x^*}(\delta_Q')$. Two cases may then be distinguished: either $x(kT_s+\floor{\epsilon/T}T)\in U_{\mathbb{R}^n}^{x^*}(\delta_Q)$, and we already know from the above result that $x(t)\in U_{\mathbb{R}^n}^{x^*}(\delta_Q')$ and, hence, $x(t)\in U_{\mathbb{R}^n}^{x^*}(\delta_B)$, for all $t\in[kT_s+\floor{\epsilon/T}T, (k+1)T_s]$, or $x(kT_s+\floor{\epsilon/T}T)\in U_{\mathbb{R}^n}^{x^*}(\delta_Q')\backslash U_{\mathbb{R}^n}^{x^*}(\delta_Q)$. The latter case implies that $x(kT_s+lT)\in U_{\mathbb{R}^n}^{x^*}(\delta_Q')\backslash U_{\mathbb{R}^n}^{x^*}(\delta_Q)$, for all $l\in\{0,...,(\floor{\epsilon/T}-1)T\}$. Therefore, it holds from the choice of $\omega_4^*$ (see \eqref{eq:lyap_std}) that $\norm{x(kT_s+\floor{\epsilon/T}T)-x^*}^2<\norm{x(kT_s)-x^*}^2-m_g(\epsilon-C/\omega_4^*)$. Selecting $\omega_5^*:=\maxs{\omega_4^*,(M_S+B)^2/(\delta_Q'-\sqrt{\delta_Q'^2-m_g(\epsilon-C/\omega_4^*)})^2}$, it holds then that $\norm{x(t)-x^*}< \sqrt{\delta_Q'^2-m_g(\epsilon-C/\omega_4^*)}+(B+M_S)/\sqrt{\omega}<\delta_Q'$, and hence $x(t)\in U_{\mathbb{R}^n}^{x^*}(\delta_B)$, for all $\omega\in(\omega^*_5,\infty)$, and $t\in[kT_s+\floor{\epsilon/T}T,(k+1)T_s]$, concluding the stability property.

\textit{Step 2.3: Practical convergence.} We know from \textit{Step 2.1.} that, for every $\omega\in(\omega_5^*,\infty)$, $t_0\in\mathbb{R}$, and $x(t_0)\in U_{\mathbb{R}^n}^{x^*}(\delta_V)$, it holds $x(t_0+T_s-\modi{t_0,T_s})\in U_{\mathbb{R}^n}^{x^*}(\delta_V')$. To prove the practical convergence property, it is thus sufficient to show that, for every $\omega\in(\omega_5^*,\infty)$, there exists a $t_1\in\mathbb{R}_{\geq 0}$ such that, for every $k_0\in\mathbb{Z}$, and $x(k_0T_s)\in U_{\mathbb{R}^n}^{x^*}(\delta_V')$, there exists a $t_1'\in[0,t_1]$ such that $x(k_0T_s+t_1')\in U_{\mathbb{R}^n}^{x^*}(\delta_Q)$. Referring to the stability property we proved in \textit{Step 2.2.}, it would indeed result $x(t)\in U_{\mathbb{R}^n}^{x^*}(\delta_B)$, for all $t\geq t_0+T_s+t_1$. Let $\omega\in(\omega^*_5,\infty)$. Note that the choice of $\omega$ ensures that (cf. \textit{Step 2.1.}) $\sqrt{\norm{x-x^*}^2-m_g(\epsilon-C/\omega)}+(B+M_S)\omega^{-0.5}<\norm{x-x^*}$, for all $x\in U_{\mathbb{R}^n}^{x^*}(\delta_V')\backslash U_{\mathbb{R}^n}^{x^*}(0.5\delta_Q)$. Therefore, there exists an $m_d\in\mathbb{R}_{>0}$ such that $\sqrt{\norm{x-x^*}^2-m_g(\epsilon-C/\omega)}+(B+M_S)\omega^{-0.5}<\norm{x-x^*}-m_d$, for all $x\in U_{\mathbb{R}^n}^{x^*}(\delta_V')\backslash U_{\mathbb{R}^n}^{x^*}(\delta_Q)$. Define $k_1:=\ceil(\!(\delta_V'-\delta_Q)/m_d)$. We will now show that, for every $k_0\in\mathbb{Z}$, and $x(k_0T_s)\in U_{\mathbb{R}^n}^{x^*}(\delta_V')$, there exists a $t_1'\in[0,k_1T_s]$ such that $x(k_0T_s+t_1')\in U^{x^*}_{\mathbb{R}^n}(\delta_Q)$. To do so, let us proceed by contradiction, and assume that there exists a $k_0'\in\mathbb{Z}$, and an $x(k_0'T_s)\in U_{\mathbb{R}^n}^{x^*}(\delta_V')$ such that $x(k_0'T_s+t)\not\in U_{\mathbb{R}^n}^{x^*}(\delta_Q)$, for all $t\in[0,k_1T_s]$. Since $x(k_0'T_s)\in U_{\mathbb{R}^n}^{x^*}(\delta_V')$, it results (see \textit{Step 2.1.}) that $x(k_0'T_s+t)\in U_{\mathbb{R}^n}^{x^*}(\delta_W)$ and, hence, $x(k_0'T_s+t)\in U_{\mathbb{R}^n}^{x^*}(\delta_W)\backslash U_{\mathbb{R}^n}^{x^*}(\delta_Q)$, for all $t\in[0,k_1T_s]$. We may thus write from \eqref{eq:lyap_std}, \eqref{eq:traj_on_std}, and \eqref{eq:bound_traj_std} that $\norm{x(\!(k_0'+k+1)T_s)-x^*}<\sqrt{\norm{x(\!(k_0'+k)T_s)-x^*}^2-m_g(\epsilon-C/\omega)}+(B+M_S)\omega^{-0.5}$, for all $k\in\{0,..,k_1-1\}$, yielding $\norm{x(\!(k_0'+k_1)T_s)-x^*}< \norm{x(k_0'T_s)-x^*}-m_d k_1$. Since $\norm{x(k_0'T_s)-x^*}\leq \delta_V'$, and since $k_1=\ceil(\!(\delta_V'-\delta_Q)/m_d)$, it results then $\norm{x(\!(k_0'+k_1)T_s)-x^*}<\delta_Q$, leading to a contradiction, and concluding the proof.
\section{Proof of Theorem \ref{lem:stop}}
\label{ap:proof_stop}
To perform the proof, we proceed in two steps. In \textit{Step 1}, we consider the extremum seeking system \eqref{eq:classical_ES}, and prove its semi-global practical uniform asymptotic stability. In \textit{Step 2}, we show that, for every $T_s\in\mathbb{R}_{>0}$, $\epsilon\in(0,T_s]$, $\omega\in\mathbb{R}_{>0}$, $\tau_0\in\mathbb{R}$, $t_0\in\mathbb{R}$, and $x_0\in\mathbb{R}^n$, the path of system \eqref{eq:classical_ES}, through $x_c(\tau(t_0)\!)=x_0$, is the same as the path of system \eqref{eq:modification_ES_intermittent_mod}, through $x(t_0)=x_0$.  We also link the convergence time of those two systems.

\textit{Step 1}. First, note that, in virtue of Assumption \ref{as:vector_fields_gen}, the Lie-bracket system associated with system \eqref{eq:classical_ES} is nothing but the gradient descent law 
\begin{equation}
\dot{\overline{x}}_c=-\rho \nabla h(\overline{x}_c).
\label{eq:associated_Lie_wait}
\end{equation}
Consider then the Lyapunov function candidate $V(\overline{x}_c)=0.5(\overline{x}_c-x^*)^T(\overline{x}_c-x^*)$. Its time-derivative, along the trajectory of system \eqref{eq:associated_Lie_wait}, is $\dot{V}(\overline{x}_c)=-\rho \nabla^Th(\overline{x}_c)(\overline{x}_c-x^*)$. In virtue of Assumption \ref{as:cost}, it holds $\nabla^Th(x)(x-x^*)>0$, for all $x\in\mathbb{R}^n\backslash\{x^*\}$, and hence $\dot{V}(\overline{x}_c)<0$, for all $\overline{x}_c\in\mathbb{R}^n\backslash\{x^*\}$. This concludes the global asymptotic stability of $x^*$  for system \eqref{eq:associated_Lie_wait}. It follows then from Lemma \ref{lem:approx_traj} that the point $x^*$ is sGPUAS for system \eqref{eq:classical_ES}, with the parameter $\omega^{-1}$. 

\textit{Step 2}. By assumption of the lemma, we know that $h(x)\neq0$, for all $x\in\mathbb{R}^n$. Accordingly, having $h_m(t,x(t)\!)=0$ implies that $\modi{t,T_s}\in[\epsilon,T_s)$ (i.e. that there is no cost measurement). Let $T_s\in\mathbb{R}_{>0}$, $\epsilon\in(0,T_s]$, $\omega\in\mathbb{R}_{>0}$, $\tau_0\in\mathbb{R}$, $t_0\in\mathbb{R}$, and $x_0\in\mathbb{R}^n$, be arbitrary, and let $k_0:=\floor{t_0/T_s}$. Note that the trajectory of system \eqref{eq:modification_ES_intermittent_mod}, through $x(t_0)=x_0$, satisfies $x(t)=x(kT_s+\epsilon)$, and $\tau(t)=\tau(kT_s+\epsilon)$, for all $k\in\mathbb{N}_{\geq k_0}$ and $t\in[kT_s+\epsilon,(k+1)T_s]$. Therefore, the path of system \eqref{eq:modification_ES_intermittent_mod}, though $x(t_0)=x_0$, is the same as the path of system \eqref{eq:classical_ES}, through $x_c(\tau(t_0)\!)=x_0$. Furthermore, if the trajectory of system \eqref{eq:classical_ES}, through $x_c(\tau(t_0)\!)=x_0$, enters in a given set at a time $t_1\in\mathbb{R}$, the trajectory of system \eqref{eq:modification_ES_intermittent_mod}, through $x(t_0)=x_0$, enters in the given set at a time $t_2\in[t_1,T_s/\epsilon t_1+T_s]$. Combining those results with the one of \textit{Step 1} concludes then the proof. 

\section{Proof of Lemma \ref{lem:approx_gradient_average}}
\label{ap:proof_gradient_average}
For the sake of compactness, let us introduce $v_i(t):=\sqrt{\omega}u_i(k_i\omega t)$, and $\tilde{f}_i(x)=f_i(h(x)\!)$, for all $i\in\{1,2,...,l\}$.

Note that, since $\omega\geq C(\epsilon')^{-1}$, it holds $k_0T_s+\epsilon'-T\in[k_0T_s,k_0T_s+\epsilon)$. Therefore, since $\epsilon'\in(0,\epsilon]$,  one has $h_m(\theta,x(\theta)\!)=h(x(\theta)\!)$, for all $\theta\in[k_0T_s+\epsilon'-T,k_0T_s+\epsilon')$. One may thus write from \eqref{eq:def_g} 
\begin{equation}
g(k_0T_s+\epsilon')=\sum_{i=1}^{l}\frac{1}{T}\int_{k_0T_s+\epsilon'-T}^{k_0T_s+\epsilon'}   \tilde{f}_i(x(\theta)\!) v_i(\theta) \, d\theta.
\label{eq:def_g_k}
\end{equation}

We are now going to apply the fundamental theorem of calculus, together with some mathematical tools, to show that the right hand-side of \eqref{eq:def_g_k} may be expressed in function of the Lie-brackets of the vector fields, plus a remainder. We will then show that the remainder may be bounded so as to obtain \eqref{eq:bound_approx_gradient_lemma}. 

First, note that the trajectory of system \eqref{eq:solution_ON_mod} is absolutely continuous on $[k_0T_s+\epsilon'-T,k_0T_s+\epsilon']$. One may thus apply the fundamental theorem of calculus to write
\begin{multline}
\tilde{f}_i(x(\theta)\!)=\tilde{f}_i(x(k_0T_s+\epsilon')\!)\\+\int_{k_0T_s+\epsilon'}^{\theta} \pazocal{D} \tilde{f}_i(x(\sigma)\!)\dot{x}(\sigma)\, d\sigma,
\label{eq:fundamental_f_i}
\end{multline}
for all $\theta\in[k_0T_s+\epsilon'-T,k_0T_s+\epsilon']$, and $i\in\{1,2,...,l\}$. Furthermore, since $\omega\geq C(\epsilon)^{-1}$, one has $\dot{x}(\sigma)=\sum_{i=1}^{l} \tilde{f}_i(x(\sigma)\!) v_i(\sigma)$, for almost every $\sigma\in[k_0T_s+\epsilon'-T,k_0T_s+\epsilon')$. It results then from \eqref{eq:fundamental_f_i}
\begin{multline}
\tilde{f}_i(x(\theta)\!)=\tilde{f}_i(x(k_0T_s+\epsilon')\!)\\+\sum_{j=1}^{l}\int_{k_0T_s+\epsilon'}^{\theta} \pazocal{L}_{\tilde{f}_j} \tilde{f}_i(x(\sigma)\!)v_j(\sigma)\, d\sigma,
\label{eq:fundamental_f_i_2}
\end{multline}
for all $\theta\in[k_0T_s+\epsilon'-T,k_0T_s+\epsilon']$ and $i\in\{1,2,...,l\}$. Substituting \eqref{eq:fundamental_f_i_2} in \eqref{eq:def_g_k} yields then
\begin{multline}
g(k_0T_s+\epsilon')=\sum_{i=1}^{l}\frac{1}{T}\tilde{f}_i(x(k_0T_s+\epsilon')\!)\times\\ \underbrace{\int_{k_0T_s+\epsilon'-T}^{k_0T_s+\epsilon'}\! v_i(\theta) d\theta}_{:=0}
+\!\!\sum_{1\leq i\leq l\atop 1\leq j\leq l}\!\frac{1}{T}\!\int_{k_0T_s+\epsilon'-T}^{k_0T_s+\epsilon'}v_i(\theta)...\\...\int_{k_0T_s+\epsilon'}^{\theta} \pazocal{L}_{\tilde{f}_j} \tilde{f}_i(h(x(\sigma)\!)\!) v_j(\sigma)\, d\sigma\,d\theta,
\label{eq:average_f_i}
\end{multline}
where we exploited the fact that the mean value of the dithers on a period is zero and that $T$ is a common period to the dithers. Applying again the fundamental theorem of calculus, one may write 
\begin{equation}
g(k_0T_s+\epsilon')\!=\frac{1}{T}\!\!\sum_{1\leq i\leq l\atop 1\leq j\leq l}\!  \pazocal{L}_{\tilde{f}_j} \tilde{f}_i(x(k_0T_s+\epsilon')\!)V_{ij}+R,\!\!
\label{eq:average_f_i_2}
\end{equation}
where we introduced 
\begin{multline}
\!\!\!\!\!\!R=\!\frac{1}{T}\!\!\!\sum_{1\leq i\leq l\atop {1\leq j\leq l\atop 1\leq m\leq l}}\!\!\int_{k_0T_s+\epsilon'-T}^{k_0T_s+\epsilon'}v_i(\theta)\!\int_{k_0T_s+\epsilon'}^{\theta} v_j(\sigma)...\\...\int_{k_0T_s+\epsilon'}^{\sigma} \pazocal{L}_{\tilde{f}_m}\pazocal{L}_{\tilde{f}_j}\tilde{f}_i(x(\mu)\!)v_m(\mu)\,d\mu\,d\sigma\,d\theta
\label{eq:def_R}
\end{multline}
and
\begin{equation}
V_{ij}=\int_{k_0T_s+\epsilon'-T}^{k_0T_s+\epsilon'}v_i(\theta)\int_{k_0T_s+\epsilon'-T}^{\theta} \!v_j(\sigma) d\sigma d\theta,
\end{equation}
and we exploited the fact that, since $T$ is a common period to the $v_i$, and since they have zero mean on a period, it holds $\int_{k_0T_s+\epsilon'-T}^{\theta} v_j(\sigma)\, d\sigma=\int_{k_0T_s+\epsilon'}^{\theta} v_j(\sigma)\, d\sigma$, for all $\theta\in\mathbb{R}$. 

To make appear the Lie-brackets in \eqref{eq:average_f_i_2}, let us divide the sum as follows
\begin{equation}
\sum_{1\leq i\leq l\atop 1\leq j\leq l}(\cdot)=\sum_{1\leq i< l\atop i< j\leq l}(\cdot)+\sum_{1\leq i\leq l\atop j=i}(\cdot)+\sum_{1< i\leq l\atop 1\leq j<i}(\cdot).
\label{eq:decomposition_sum}
\end{equation}
Performing an integration by part, while remembering that $\int_{k_0T_s+\epsilon'-T}^{k_0T_s+\epsilon'}v_i(\theta)\, d\theta=0$, for all $i\in\{1,2,...,l\}$, one may write, $V_{ij}=-V_{ji}$, for all $i,j\in\{1,2,...,l\}$. Exploiting this result in the first and second term of the right-side of \eqref{eq:decomposition_sum}, one obtains then from \eqref{eq:average_f_i_2}
\begin{multline}
\!\!\!\!\!g(k_0T_s+\epsilon')\!=\!-\frac{1}{T}\!\sum_{1\leq i< l\atop i<j\leq l}\! \pazocal{L}_{\tilde{f}_j} \tilde{f}_i(x(k_0T_s+\epsilon')\!)V_{ji}\\\!\!\!+\frac{1}{T}\sum_{1< i\leq l\atop 1\leq j< i} \pazocal{L}_{\tilde{f}_j} \tilde{f}_i(x(k_0T_s+\epsilon')\!)V_{ij}+R.
\label{eq:average_f_i_3}
\end{multline}
Switching the indices $i$ and $j$ in the second term of the right-side of \eqref{eq:average_f_i_3} and exploiting the fact that $\sum_{1<j\leq l\atop 1\leq i< j}a_ib_j=\sum_{1\leq i< l\atop i< j\leq i}a_ib_j$ yields
\begin{multline}
g(k_0T_s+\epsilon')\!=\!\sum_{1\leq i< l\atop i<j\leq l}\!\left(\pazocal{L}_{\tilde{f}_i} \tilde{f}_j(x(k_0T_s+\epsilon')\!)\right.\\\left.-\pazocal{L}_{\tilde{f}_j} \tilde{f}_i(x(k_0T_s+\epsilon')\!)\right)\gamma_{ij}\!+\!R,
\label{eq:average_f_i_4}
\end{multline}
where we exploited the definition of $\gamma_{ij}$ given in \eqref{eq:def_gamma_ij}. This allows us to conclude that 
\begin{equation}
g(k_0T_s+\epsilon')\!=\!\sum_{1\leq i< l\atop i<j\leq l}\![\tilde{f}_i,\tilde{f}_j](x(k_0T_s+\epsilon')\!) \gamma_{ij}+ R.
\label{eq:bound_g_k}
\end{equation}
We know from Assumption \ref{as:vector_fields_gen} that 
\begin{equation}
\sum_{1\leq i< l\atop i<j\leq l}[\tilde{f}_i,\tilde{f}_j](h(x)\!) \gamma_{ij}=-\rho \nabla h(x),
\end{equation}
for all $x\in\mathbb{R}^n$. Furthermore, since $x(t)\in\mathcal{W}$, for all $t\in[k_0T_s+\epsilon'-T,k_0T_s+\epsilon']$, and since the $\tilde{f}_i(x)$ are vector fields of class $C^2$, there exists an $M\in\mathbb{R}_{>0}$ such that $\norm{\pazocal{L}_{\tilde{f}_m}\pazocal{L}_{\tilde{f}_j}\tilde{f}_i(x(t)\!)}<M$, for all $t\in[k_0T_s+\epsilon'-T,k_0T_s+\epsilon']$ and $i,j,m\in\{1,2,...,l\}$. Accordingly, remembering that $T:=C\omega^{-1}$, and that $\sup_{t\in\mathbb{R}}\vert v_i(t)\vert \leq \sqrt{\omega}$, for all $i\in\{1,2,...,l\}$, one may bound the remainder $R$, defined in \eqref{eq:def_R}, as follows
\begin{equation}
\norm{R}\leq l^3 M C^2 \omega^{-0.5}.
\end{equation} 
This allows us to obtain \eqref{eq:bound_approx_gradient_lemma} from \eqref{eq:bound_g_k}, concluding the proof. 
\section{Proof of Theorem \ref{lem:stability_gradient_average}}
\label{ap:proof_stability_gradient_average}
For the sake of compactness, let us introduce $C\!=\!2\pi \textbf{LCM}(k_1^{-1},k_2^{-1},...,k_l^{-1})$, $T\!=\!C\omega^{-1}$, and $k_T\!=\!\floor{\epsilon'/T}$. Let $\delta_V\in\mathbb{R}_{>0}$, and $\delta_B\in\mathbb{R}_{>0}$, be arbitrary, but fixed. Select then $\delta_V'\in(\delta_V,\infty)$, $\delta_Q'\in(0,\mins{\delta_B,\delta_V})$, $\delta_Q\in(0,\delta_Q')$, and $\delta_W\in(\delta_V',\infty)$.

To perform the proof, we refer to Definition \ref{def:sgpuas}, and show that its three properties are enforced by system \eqref{eq:solution_ON_mod}. As in the proof of Theorem \ref{lem:classical_source}, we consider the first transmission period separately. We then use recursive arguments to handle the next periods.

\textit{Step 1: Boundedness.} Define $\omega^*_1=C(\epsilon')^{-1}$. Note that the dynamics of system \eqref{eq:solution_ON_mod} is the same as the one of system \eqref{eq:classical_source}, for all $\omega\in[\omega^*_1,\infty)$, $k\in \mathbb{Z}$ and $t\in[kT_s,kT_s+k_T T)$. It follows thus from the proof of Theorem \ref{lem:classical_source} that there exists an $\omega^*_2\in(\omega^*_1,\infty)$ such that, for all $\omega\in(\omega^*_2,\infty)$, $k\in\mathbb{Z}$, and $t_0\in[kT_s,kT_s+k_T T)$, if $x(t_0)\in U_{\mathbb{R}^n}^{x^*}(\delta_V)$, then $x(t)\in U_{\mathbb{R}^n}^{x^*}(0.5\delta_V+0.5\delta_V')$, for all $t\in[t_0,kT_s+k_T T]$. Furthermore, Assumptions \ref{as:cost} and \ref{as:vector_fields_gen} ensure the existence of an $M\in\mathbb{R}_{>0}$ such that $\norm{f_i(h_m(t,x)\!)}<M$, for all $i\in\{1,2,...,l\}$, $t\in\mathbb{R}$, and $x\in U^{x^*}_{\mathbb{R}^n}(0.5\delta_V+0.5\delta_V')$. Define $\rho^*_2(\omega)=0.5(\delta_V'-\delta_V)/(\sqrt{\omega}lM(T_s-k_TT)\!)$. It results thus that, for all $k\in\mathbb{Z}$, $t_0\in[kT_s+k_T T,(k+1)T_s)$, $\omega\in(\omega^*_2,\infty)$, and $\rho_2\in(0,\rho_2^*(\omega)\!)$, if $x(t)\in U_{\mathbb{R}^n}^{x^*}(0.5\delta_V+0.5\delta_V')$, for all $t\in[t_0-T,t_0]$, then $\norm{g(\floor{t/T_s}T_s+k_TT)}<\sqrt{\omega}lM$ and, hence, $x(t)\in U_{\mathbb{R}^n}^{x^*}(\delta_V')$, for all $t\in[t_0,(k+1)T_s]$. One may thus conclude that, for all $t_0\in\mathbb{R}$,  $\omega\in(\omega^*_2,\infty)$, and $\rho_2\in(0,\rho_2^*(\omega)\!)$, if $x(t)\in U_{\mathbb{R}^n}^{x^*}(\delta_V)$, for all $t\in[t_0-T,t_0]$, then it holds $x(t)\in  U_{\mathbb{R}^n}^{x^*}(\delta_V')$, for all $t\in[t_0,t_0+T_s-\modi{t_0,T_s}]$. Define now  $V(x)=0.5(x-x^*)^T(x-x^*)$. We know from Proposition 1 and Lemma 2 in \cite{Lab22} that there exist an $\omega_3^*\in(\omega^*_2,\infty)$, and an $M_V\in\mathbb{R}_{>0}$, such that, for every $k_0\in\mathbb{Z}$, $k\in\{0,...,k_T-1\}$, $x_0\in U_{\mathbb{R}^n}^{x^*}(\delta_V')$, and $\omega\in(\omega^*_3,\infty)$, the trajectory of system \eqref{eq:solution_ON_mod}, through $x(k_0T_s+kT)=x_0$, satisfies
\begin{multline}
\norm{x(t)-x(k_0T_s+kT)}\\<\mins{\delta_W-\delta_V',0.5\mins{\delta_B,\delta_V}-0.5\delta_Q'},
\label{eq:traj_mod_lb}
\end{multline}
for all $t\in[k_0T_s+kT,k_0T_s+(k+1)T]$ and
\begin{multline}
V(x(k_0T_s+(k+1)T)\!)\leq V(x(k_0T_s+kT)\!)+TM_V\omega^{-0.5}\\-\rho T\nabla^T h(x(k_0T_s+kT)\!)(x(k_0T_s+kT)-x^*).
\label{eq:lya_mod}
\end{multline}

In virtue of Assumption \ref{as:cost}, it holds $\nabla^T h(x) (x-x^*)>0$, for all $x\in\mathbb{R}^n\backslash\{x^*\}$. Let then $m_g:=\inf_{x\in U_{\mathbb{R}}^{x^*}(\delta_V')\backslash U_{\mathbb{R}}^{x^*}(\delta_Q)} \nabla^T h(x) (x-x^*)$ and define $\omega_4^*=\maxs{\omega_3^*,M_V^2/(\rho^2m_g^2)}$. 

Let $x_0\in U_{\mathbb{R}^n}^{x^*}(\delta_V')$, $k_0\in\mathbb{Z}$, $k\in\{0,...,k_T-1\}$, and $\omega\in(\omega^*_4,\infty)$. We already know from \eqref{eq:traj_mod_lb} that the trajectory of system \eqref{eq:solution_ON_mod}, through $x(k_0T_s+kT)=x_0$, satisfies $x(t)\in U_{\mathbb{R}^n}^{x^*}(\delta_W)$, for all $t\in[k_0T_s+kT,k_0T_s+(k+1)T]$. Two cases may then be distinguished. Either $x(k_0T_s+kT)\in U_{\mathbb{R}^n}^{x^*}(\delta_V')\backslash U_{\mathbb{R}^n}^{x^*}(\delta_Q')$ or $x(k_0T_s+kT)\in U_{\mathbb{R}^n}^{x^*}(\delta_Q')$. In the former case, it results from \eqref{eq:lya_mod}, and the choice of $\omega^*_4$, that $V(x(k_0T_s+(k+1)T)\!)<V(x(k_0T_s+kT)\!)$, and hence $x(k_0T_s+(k+1)T)\in U_{\mathbb{R}^n}^{x^*}(\delta_V')$.  In the latter case, we know from \eqref{eq:traj_mod_lb} that $x(k_0T_s+(k+1)T)\in U_{\mathbb{R}^n}^{x^*}(0.5\delta_Q'+0.5\mins{\delta_B,\delta_V})$, and hence $x(k_0T_s+(k+1)T)\in U_{\mathbb{R}^n}^{x^*}(\delta_V')$. One may thus repeat the same reasoning to conclude that, for every $k_0\in\mathbb{Z}$, $x_0\in U_{\mathbb{R}^n}^{x^*}(\delta_V')$, and $\omega\in(\omega^*_4,\infty)$, the trajectory of system \eqref{eq:solution_ON_mod}, through $x(k_0T_s)=x_0$, satisfies $x(t)\in U_{\mathbb{R}^n}^{x^*}(\delta_W)$, for all $t\in[k_0T_s,k_0T_s+k_TT]$, and $x(k_0T_s+k_TT)\in U_{\mathbb{R}^n}^{x^*}(\delta_V')$. 

Therefore, Lemma \ref{lem:approx_gradient_average} ensures that, for all $k_0\in\mathbb{Z}$, and $\omega\in(\omega^*_4,\infty)$, if $x(k_0T_s)\in U_{\mathbb{R}^n}^{x^*}(\delta_V')$, then  it holds $g(k_0T_s+k_TT)=-\rho \nabla h(x(k_0T_s+k_TT)\!)+R_T$, with $\norm{R_T}<M\omega^{-0.5}$ and $M$ such that $\norm{\pazocal{L}_{f_m}\pazocal{L}_{f_j}f_i(h(x)\!)}<M l^{-3}C^2$, for all $x\in U_{\mathbb{R}^n}^{x^*}(\delta_W)$ and $i,j,m\in\{1,2,...,l\}$. For all $k_0\in\mathbb{Z}$, $x_0\in U_{\mathbb{R}^n}^{x^*}(\delta_V')$, and $\omega\in(\omega^*_4,\infty)$, the trajectory of system \eqref{eq:solution_ON_mod}, through $x(k_0T_s)=x_0$, follows thus
\begin{multline}
\!\!\!\!\!\!\norm{x(t)-x^*}^2\!=\!\norm{x(k_0T_s+k_TT)-x^*}^2-\rho_2(t-k_0T_s-k_TT)\\ \times[2(x(k_0T_s+k_TT)-x^*)^T (\rho\nabla h(x(k_0T_s+k_TT)\!)-R_T) \\-\rho_2(t-k_0T_s-k_TT)(\rho^2\norm{\nabla h(x(k_0T_s+k_TT)\!)}^2+\norm{R_T}^2\\-2\rho\nabla^T h(x(k_0T_s+k_TT)\!)R_T)],
\end{multline}
for all $t\in[k_0T_s+k_TT,(k_0+1)T_s]$, which may be bounded as follows
\begin{multline}
\!\!\!\!\!\!\norm{x(t)-x^*}^2\!\leq\! \norm{x(k_0T_s+k_TT)-x^*}^2-\rho_2(t-k_0T_s-k_TT)\\ \times[2\rho(x(k_0T_s+k_TT)-x^*)^T\nabla h(x(k_0T_s+k_TT)\!)-2M\omega^{-0.5}\delta_V'\\-\rho_2 (T_s-k_TT)(M_G^2+M^2\omega^{-1}+2M_GM\omega^{-0.5})],
\label{eq:average_decrease_nd_lie}
\end{multline}
for all $t\in[k_0T_s+k_TT,(k_0+1)T_s]$, with $M_G:=\rho\sup_{x\in U_{\mathbb{R}}^{x^*}(\delta_V')} \norm{\nabla h(x)}$. Select now
\begin{equation}
\omega^*_5:=\maxs{\omega_4^*,4M^2\delta_V^2/(\rho m_g)^2}
\end{equation}
and
\begin{multline}
\rho_3^*(\omega):=\min\left\lbrace\rho^*_2(\omega),\right.\\\left.\rho m_g/(T_s(M_G^2+M^2/\omega^*+2M_GM/\sqrt{\omega^*})\!)\right\rbrace.
\end{multline} 
It results then from \eqref{eq:average_decrease_nd_lie} that, if $x(k_0T_s+k_TT)\in U_{\mathbb{R}}^{x^*}(\delta_V')\backslash U_{\mathbb{R}}^{x^*}(\delta_Q)$, then $\norm{x(t)-x^*}<\norm{x(k_0T_s+k_TT)-x^*}$ and, hence, $x(t)\in U_{\mathbb{R}}^{x^*}(\delta_V')$, for all $t\in[k_0T_s+k_TT,(k_0+1)T_s]$, $\omega\in(\omega^*_5,\infty)$ and $\rho_2\in(0,\rho^*_3(\omega)\!)$.  

Furthermore, define
\begin{multline}
\rho^*_4(\omega)\!\\=\!\min\left\lbrace\rho^*_3(\omega),\frac{0.5(\mins{\delta_B,\delta_V}-\delta_Q)}{T_s\!\left(\rho\sup_{x\in U_{\mathbb{R}}^{x^*}(\delta_B)}\!\norm{\nabla h(x)}\!+\!M/\sqrt{\omega}\right)\!}\!\right\rbrace\!.
\end{multline}  
We may write from \eqref{eq:solution_ON_mod} 
\begin{multline}
\!\!\!\!\!\!	\norm{x(t)-x^*}\leq \norm{x(k_0T_s+k_TT)-x^*}+\rho_2(t-k_0T_s-k_TT)\\ \times(\rho\norm{\nabla h(x(k_0T_s+k_TT)\!)}+\norm{R_T}),
\label{eq:norm_traj_mod}
\end{multline}
for all $t\in[k_0T_s+k_TT,(k_0+1)T_s]$. Therefore, if $x(k_0T_s+k_TT)\in U_{\mathbb{R}}^{x^*}(\delta_Q)$, then $x(t)\in U_{\mathbb{R}}^{x^*}(0.5\delta_Q+0.5\mins{\delta_B+\delta_V})$, and hence $x(t)\in U_{\mathbb{R}}^{x^*}(\delta_V')$, for all  $t\in[k_0T_s+k_TT,(k_0+1)T_s]$, $\omega\in(\omega^*_5,\infty)$, and $\rho_2\in(0,\rho^*_4(\omega)\!)$. 

To sum up, we just proved that, for all $k_0\in\mathbb{Z}$, $x_0\in U_{\mathbb{R}^n}^{x^*}(\delta_V')$, $\omega\in(\omega^*_5,\infty)$, and $\rho_2\in(0,\rho^*_4(\omega)\!)$, the trajectory of system \eqref{eq:solution_ON_mod}, through $x(k_0T_s)=x_0$, satisfies $x(t)\in U_{\mathbb{R}^n}^{x^*}(\delta_W)$, for all $t\in[k_0T_s,(k_0+1)T_s]$, and $x(\!(k_0+1)T_s)\in U_{\mathbb{R}^n}^{x^*}(\delta_V')$. 

Combining all the results, one may thus conclude that, for all $t_0\in\mathbb{R}$, $\omega\in(\omega^*_5,\infty)$, and $\rho_2\in(0,\rho^*_4(\omega)\!)$, if $x(t)\in U_{\mathbb{R}}^{x^*}(\delta_V)$, for all $t\in[t_0-T,t_0]$, then $x(t)\in  U_{\mathbb{R}}^{x^*}(\delta_W)$, for all $t\geq t_0$, concluding the boundedness property.

\textit{Step 2: Stability Property.} Let us follow a reasoning similar to the one used in \textit{Step 1}. Note that the proof of Theorem \ref{lem:classical_source} ensures the existence of an $\omega^*\in(\omega^*_5,\infty)$ such that, for all $\omega\in(\omega^*,\infty)$, $k\in\mathbb{Z}$, and $t_0\in[kT_s,kT_s+k_TT)$, if $x(t_0)\in U_{\mathbb{R}^n}^{x^*}(\delta_Q)$, then $x(t)\in U_{\mathbb{R}^n}^{x^*}(0.5\delta_Q+0.5\delta_Q')$, for all $t\in[t_0,kT_s+k_T T]$. Furthermore, in virtue of Assumptions \ref{as:cost} and \ref{as:vector_fields_gen}, there exists an $M\in\mathbb{R}_{>0}$ such that $\norm{f_i(h_m(t,x)\!)}<M$, for all $i\in\{1,2,...,l\}$, $t\in\mathbb{R}$, and $x\in U^{x^*}_{\mathbb{R}^n}(0.5\delta_Q+0.5\delta_Q')$. Select $\rho^*_5(\omega)=\mins{\rho^*_4(\omega),(0.5\delta_Q'-0.5\delta_Q)/(\sqrt{\omega}lM)(T_s-k_TT)}$. For all $k\in\mathbb{Z}$, $t_0\in[kT_s+k_TT,(k+1)T_s)$, $\omega\in(\omega^*,\infty)$, and $\rho_2\in(0,\rho_5^*(\omega)\!)$, if $x(t)\in U_{\mathbb{R}^n}^{x^*}(0.5\delta_Q+0.5\delta_Q')$, for all $t\in[t_0-T,t_0]$, it holds then $\norm{g(\floor{t/T_s}T_s+k_TT)}<\sqrt{\omega}lM$ and, hence, $x(t)\in U_{\mathbb{R}^n}^{x^*}(\delta_Q')$, for all $t\in[t_0,(k+1)T_s]$. One may thus conclude that, for all $t_0\in\mathbb{R}$,  $\omega\in(\omega^*,\infty)$, and $\rho_2\in(0,\rho_5^*(\omega)\!)$, if $x(t)\in U_{\mathbb{R}^n}^{x^*}(\delta_Q)$, for all $t\in[t_0-T,t_0]$, then it holds $x(t)\in  U_{\mathbb{R}^n}^{x^*}(\delta_Q')$, for all $t\in[t_0,t_0+T_s-\modi{t_0,T_s}]$. We will now show that, for all $k_0\in\mathbb{Z}$, $x_0\in U_{\mathbb{R}}^{x^*}(0.5\delta_Q'+0.5\mins{\delta_V,\delta_B})$, $\omega\in(\omega^*,\infty)$, and $\rho_2\in(0,\rho^*_5(\omega)\!)$, the trajectory of system \eqref{eq:solution_ON_mod}, through $x(k_0T_s)=x_0$, fulfills $x(t)\in U_{\mathbb{R}}^{x^*}(\delta_B)$, for all $t\geq k_0T_s$. Let $k_0\in\mathbb{Z}$, $k\in\{0,...,k_T-1\}$, $x(k_0T_s+kT)\in U_{\mathbb{R}}^{x^*}(0.5\delta_Q'+0.5\mins{\delta_V,\delta_B})$, $\omega\in(\omega^*,\infty)$, and $\rho_2\in(0,\rho^*_5(\omega)\!)$. Note that, since $\delta_Q'\in(0,\delta_V)$, it results from \eqref{eq:traj_mod_lb} that $x(t)\in U_{\mathbb{R}}^{x^*}(\mins{\delta_B,\delta_V})$ and, hence,  $x(t)\in U_{\mathbb{R}}^{x^*}(\delta_B)$, for all $t\in[k_0T_s+kT,k_0T_s+(k+1)T]$. We also know from \eqref{eq:traj_mod_lb} that, if $x(k_0T_s+kT)\in U_{\mathbb{R}}^{x^*}(\delta_Q')$, then $x(k_0T_s+(k+1)T)\in U_{\mathbb{R}}^{x^*}(0.5\delta_Q'+0.5\mins{\delta_B,\delta_V})$. Furthermore, if $x(k_0T_s+kT)\in U_{\mathbb{R}}^{x^*}(0.5\delta_Q'+0.5\mins{\delta_B,\delta_V})\backslash U_{\mathbb{R}}^{x^*}(\delta_Q')$, it results from \eqref{eq:lya_mod} that $\norm{x(k_0T_s+(k+1)T)-x^*}<\norm{x(k_0T_s+kT)-x^*}$ and, hence, $x(k_0T_s+(k+1)T)\in U_{\mathbb{R}}^{x^*}(0.5\delta_Q'+0.5\mins{\delta_B,\delta_V})$. One may thus repeat the reasoning to conclude that, for all $k_0\in\mathbb{Z}$, $x_0\in U_{\mathbb{R}}^{x^*}(0.5\delta_Q'+0.5\mins{\delta_V,\delta_B})$, $\omega\in(\omega^*,\infty)$, and $\rho_2\in(0,\rho^*_5(\omega)\!)$, the trajectory of system \eqref{eq:solution_ON_mod}, through $x(k_0T_s)=x_0$, fulfills i) $x(t)\in  U_{\mathbb{R}}^{x^*}(\delta_B)$, for all $t\in[k_0T_s,k_0T_s+k_TT]$, and ii) $x(k_0T_s+k_TT)\in  U_{\mathbb{R}}^{x^*}(0.5\delta_Q'+0.5\mins{\delta_B,\delta_V})$. Two cases may then be distinguished: either $x(k_0T_s+k_TT)\in U_{\mathbb{R}}^{x^*}(\delta_Q)$, or $x(k_0T_s+k_TT)\in U_{\mathbb{R}}^{x^*}(0.5\delta_Q'+0.5\mins{\delta_B,\delta_V})\backslash U_{\mathbb{R}}^{x^*}(\delta_Q)$. In the former case, it holds from \eqref{eq:norm_traj_mod} that $x(t)\in U_{\mathbb{R}}^{x^*}(0.5\delta_Q'+0.5\mins{\delta_B,\delta_V})$ and, hence, $x(t)\in U_{\mathbb{R}}^{x^*}(\delta_B)$, for all $t\in[k_0T_s+k_TT,(k_0+1)T_s]$, $\omega\in(\omega^*,\infty)$, and $\rho_2\in(0,\rho^*_5(\omega)\!)$. In the latter case, it results from \eqref{eq:average_decrease_nd_lie} that $\norm{x(t)-x^*}<\norm{x(k_0T_s+k_TT)-x^*}$ and, hence, $x(t)\in U_{\mathbb{R}}^{x^*}(0.5\delta_Q'+0.5\mins{\delta_B,\delta_V})$, for all $t\in[k_0T_s+k_TT,(k_0+1)T_s]$, $\omega\in(\omega^*,\infty)$, and $\rho_2\in(0,\rho^*_5(\omega)\!)$. In both cases, one obtains $x(\!(k_0+1)T_s)\in U_{\mathbb{R}}^{x^*}(0.5\delta_Q'+0.5\mins{\delta_B,\delta_V})$, for all $\omega\in(\omega^*,\infty)$, and $\rho_2\in(0,\rho^*_5(\omega)\!)$. The stability property may thus be concluded by iterating the reasoning with $k_0=k_0+T_s$ and $x_0=x(\!(k_0+1)T_s)$.

\textit{Step 3: Practical Convergence.} We proved in \textit{Step 1} that, for all $t_0\in\mathbb{R}$,  $\omega\in(\omega^*,\infty)$, and $\rho_2\in(0,\rho_5^*(\omega)\!)$, if $x(t)\in U_{\mathbb{R}^n}^{x^*}(\delta_V)$, for all $t\in[t_0-T,t_0]$, then it holds $x(t_0+T_s-\modi{t_0,T_s}\!)\in  U_{\mathbb{R}^n}^{x^*}(\delta_V')$. Referring to the stability property we proved in \textit{Step 2}, it is thus sufficient to show that, for every $\omega\in(\omega^*,\infty)$, and $\rho_2\in(0,\rho^*_5(\omega)\!)$, there exists a $k_1\in\mathbb{N}_{>0}$ such that, for every $k_0\in\mathbb{Z}$, and $x(k_0T_s)\in U_{\mathbb{R}}^{x^*}(\delta_V')$, there exists a $k_1'\in\{0,...,k_1-1\}$ and a $k_1''\in\{0,...,k_T\}$ such that  $x(\!(k_0+k_1')T_s+k_1''T)\in U_{\mathbb{R}}^{x^*}(0.5\delta_Q'+\mins{\delta_B,\delta_V})$. The practical convergence property would indeed be concluded with $t_1:=(1+k_1)T_s$. Let $\omega\in(\omega^*,\infty)$ and $\rho_2\in(0,\rho^*_5(\omega)\!)$. We already know from \textit{Step 1} that, for every $k_0\in\mathbb{Z}$, if $x(k_0T_s)\in U_{\mathbb{R}}^{x^*}(\delta_V')$, then $x(\!(k_0+k)T_s+k'T)\in U_{\mathbb{R}}^{x^*}(\delta_V')$, for every $k'\in\{0,...,k_T\}$ and $k\in\mathbb{N}$. Moreover, it follows from \eqref{eq:lya_mod} that there exists an $m_d\in\mathbb{R}_{>0}$ such that, for every $k_0\in\mathbb{Z}$, and $k'\in\{0,...,k_T-1\}$, if $x(k_0T_s+k'T)\in U_{\mathbb{R}}^{x^*}(\delta_V')\backslash U_{\mathbb{R}}^{x^*}(\delta_Q)$, then $\norm{x(k_0T_s+(k'+1)T)-x^*}<\norm{x(k_0T_s+k'T)-x^*}-m_d$. Similarly, we may deduce from \eqref{eq:average_decrease_nd_lie} the existence of an $m_d'\in\mathbb{R}_{>0}$ such that, for every $k_0\in\mathbb{Z}$, if  $x(k_0T_s+k_TT)\in U_{\mathbb{R}}^{x^*}(\delta_V')\backslash U_{\mathbb{R}}^{x^*}(\delta_Q)$, then $\norm{x(k_0+1)T_s)-x^*}<\norm{x(k_0T_s+k_TT)-x^*}-m_d'$. To show that $k_1^*:=\ceil\left(\!(\delta_V'-\delta_Q)/(m_dk_T+m_d')\!\right)+1$ is a suited value of $k_1$, let us proceed by contradiction. Namely, let us assume that there exist a $k_0'\in\mathbb{Z}$, and an $x(k_0'T_s)\in U_{\mathbb{R}^n}^{x^*}(\delta_V')$, such that $x(\!(k_0'+k)T_s+k'T)\not\in U_{\mathbb{R}^n}^{x^*}(0.5\delta_Q'+\mins{\delta_B,\delta_V})$, for all $k\in\{0,...,k_1^*-1\}$ and $k'\in\{0,...,k_T\}$. We may then write that $\norm{x(\!(k_0'+k^*_1-1)T_s+k_TT)-x^*}<\norm{x(k_0'T_s)-x^*}-k^*_1k_Tm_d-(k_1^*-1)m_d'$, yielding $\norm{x(\!(k_0'+k^*_1-1)T_s+k_TT)-x^*}\leq \delta_Q$. This leads to a contradiction and concludes thus the proof.
\bibliographystyle{IEEEtran}
\bibliography{biblio}

\begin{thebibliography}{10}
\providecommand{\url}[1]{#1}
\csname url@samestyle\endcsname
\providecommand{\newblock}{\relax}
\providecommand{\bibinfo}[2]{#2}
\providecommand{\BIBentrySTDinterwordspacing}{\spaceskip=0pt\relax}
\providecommand{\BIBentryALTinterwordstretchfactor}{4}
\providecommand{\BIBentryALTinterwordspacing}{\spaceskip=\fontdimen2\font plus
\BIBentryALTinterwordstretchfactor\fontdimen3\font minus
  \fontdimen4\font\relax}
\providecommand{\BIBforeignlanguage}[2]{{%
\expandafter\ifx\csname l@#1\endcsname\relax
\typeout{** WARNING: IEEEtran.bst: No hyphenation pattern has been}%
\typeout{** loaded for the language `#1'. Using the pattern for}%
\typeout{** the default language instead.}%
\else
\language=\csname l@#1\endcsname
\fi
#2}}
\providecommand{\BIBdecl}{\relax}
\BIBdecl

\bibitem{Att19}
K.~T. Atta and M.~Guay, ``Adaptive amplitude fast proportional integral phasor
  extremum seeking control for a class of nonlinear system,'' \emph{Journal of
  Process Control}, vol.~83, pp. 147--154, 2019.

\bibitem{Hun14}
B.~Hunnekens, M.~Haring, N.~van~de Wouw, and H.~Nijmeijer, ``A dither-free
  extremum-seeking control approach using 1st-order least-squares fits for
  gradient estimation,'' in \emph{53rd IEEE Conference on Decision and
  Control}.\hskip 1em plus 0.5em minus 0.4em\relax IEEE, 2014, pp. 2679--2684.

\bibitem{Gua17}
M.~Guay and D.~Dochain, ``A proportional-integral extremum-seeking controller
  design technique,'' \emph{Automatica}, vol.~77, pp. 61--67, 2017.

\bibitem{Tol17}
S.~F. {Toloue} and M.~{Moallem}, ``Multivariable sliding-mode extremum seeking
  control with application to mppt of an alternator-based energy conversion
  system,'' \emph{IEEE Transactions on Industrial Electronics}, vol.~64, no.~8,
  pp. 6383--6391, 2017.

\bibitem{Lab18}
C.~Labar, J.~Feiling, and C.~Ebenbauer, ``Gradient-based extremum seeking:
  Performance tuning via lie bracket approximations,'' in \emph{2018 European
  Control Conference (ECC)}.\hskip 1em plus 0.5em minus 0.4em\relax IEEE, 2018,
  pp. 2775--2780.

\bibitem{Sut19}
R.~Suttner, ``Extremum seeking control with an adaptive dither signal,''
  \emph{Automatica}, vol. 101, pp. 214 -- 222, 2019.

\bibitem{Gru20}
V.~Grushkovskaya and C.~Ebenbauer, ``Extremum seeking control of nonlinear
  dynamic systems using lie bracket approximations,'' \emph{International
  Journal of Adaptive Control and Signal Processing}, 2020.

\bibitem{Lab19}
C.~Labar, E.~Garone, M.~Kinnaert, and C.~Ebenbauer, ``Newton-based extremum
  seeking: A second-order lie bracket approximation approach,''
  \emph{Automatica}, vol. 105, pp. 356 -- 367, 2019.

\bibitem{Moa10}
W.~H. {Moase}, C.~{Manzie}, and M.~J. {Brear}, ``Newton-like extremum-seeking
  for the control of thermoacoustic instability,'' \emph{IEEE Transactions on
  Automatic Control}, vol.~55, no.~9, pp. 2094--2105, 2010.

\bibitem{Gro16}
M.~Gro{\ss}bichler, R.~Schmied, P.~Polterauer, H.~Waschl, and L.~del Re, ``A
  robustified newton based extremum seeking for engine optimization,'' in
  \emph{2016 American Control Conference}.\hskip 1em plus 0.5em minus
  0.4em\relax IEEE, 2016, pp. 3280--3285.

\bibitem{Gha14}
A.~Ghaffari, M.~Krsti{\'c}, and S.~Seshagiri, ``Power optimization and control
  in wind energy conversion systems using extremum seeking,'' \emph{IEEE
  Transactions on Control Systems Technology}, vol.~22, no.~5, pp. 1684--1695,
  2014.

\bibitem{Rot17}
M.~A. Rotea, ``Logarithmic power feedback for extremum seeking control of wind
  turbines,'' \emph{IFAC-PapersOnLine}, vol.~50, no.~1, pp. 4504--4509, 2017.

\bibitem{Raf18}
S.~M. Rafaat, R.~Hussein \emph{et~al.}, ``Power maximization and control of
  variable-speed wind turbine system using extremum seeking,'' \emph{Journal of
  Power and Energy Engineering}, vol.~6, no.~01, p.~51, 2018.

\bibitem{Bas09}
G.~Bastin, D.~Ne{\v{s}}i{\'c}, Y.~Tan, and I.~Mareels, ``On extremum seeking in
  bioprocesses with multivalued cost functions,'' \emph{Biotechnology
  progress}, vol.~25, no.~3, pp. 683--689, 2009.

\bibitem{Hal19}
W.~{Halter}, S.~{Michalowsky}, and F.~{Allgöwer}, ``Extremum seeking for
  optimal enzyme production under cellular fitness constraints,'' in \emph{2019
  18th European Control Conference (ECC)}, 2019, pp. 2159--2164.

\bibitem{Dew17}
L.~{Dewasme}, C.~G. {Feudjio Letchindjio}, I.~T. {Zuniga}, and A.~{Vande
  Wouwer}, ``Micro-algae productivity optimization using extremum-seeking
  control,'' in \emph{2017 25th Mediterranean Conference on Control and
  Automation (MED)}, 2017, pp. 672--677.

\bibitem{Gua14}
M.~{Guay} and D.~J. {Burns}, ``A comparison of extremum seeking algorithms
  applied to vapor compression system optimization,'' in \emph{2014 American
  Control Conference}, 2014, pp. 1076--1081.

\bibitem{Koe14}
J.~P. Koeln and A.~G. Alleyne, ``Optimal subcooling in vapor compression
  systems via extremum seeking control: Theory and experiments,''
  \emph{International journal of refrigeration}, vol.~43, pp. 14--25, 2014.

\bibitem{Man20}
F.~Mandić, N.~Mišković, and I.~Lončar, ``Underwater acoustic source seeking
  using time-difference-of-arrival measurements,'' \emph{IEEE Journal of
  Oceanic Engineering}, vol.~45, no.~3, pp. 759--771, 2020.

\bibitem{Xu19}
S.~Xu, Y.~Wang, D.~Xu, X.~Zhu, and H.~Chen, ``A review on source seeking
  control and its application to wheeled mobile robots,'' in \emph{2019 3rd
  Conference on Vehicle Control and Intelligence (CVCI)}, 2019, pp. 1--5.

\bibitem{Ila20}
I.~Azzollini, N.~Mimmo, and L.~Marconi, ``An extremum seeking approach to
  search and rescue operations in avalanches using arva,'' in \emph{Proceedings
  of the 21$^{st}$ IFAC World Congress}, 2020.

\bibitem{Sil17}
M.~Silvagni, A.~Tonoli, E.~Zenerino, and M.~Chiaberge, ``Multipurpose uav for
  search and rescue operations in mountain avalanche events,'' \emph{Geomatics,
  Natural Hazards and Risk}, vol.~8, no.~1, pp. 18--33, 2017.

\bibitem{Dur_thesis}
H.-B. D\"urr, ``Constrained extremum seeking: A {L}ie bracket and singular
  pertubation approach,'' PhD thesis, Stuttgart University, 2015.

\bibitem{Haz19}
L.~Hazeleger, R.~Beerens, and N.~van~de Wouw, ``A sampled-data extremum-seeking
  approach for accurate setpoint control of motion systems with friction,'' in
  \emph{Proceedings of the 11$^{th}$ IFAC symposium on Nonlinear Control
  Systems}, 2019, pp. 801--806.

\bibitem{Kho13}
S.~Z. Khong, D.~Ne{\v{s}}i{\'c}, Y.~Tan, and C.~Manzie, ``Trajectory-based
  proofs for sampled-data extremum seeking control,'' in \emph{2013 American
  Control Conference}.\hskip 1em plus 0.5em minus 0.4em\relax IEEE, 2013, pp.
  2751--2756.

\bibitem{Pre20}
U.~Premaratne, S.~Halgamuge, Y.~Tan, and I.~M. Mareels, ``Extremum seeking
  control with sporadic packet transmission for networked control systems,''
  \emph{IEEE Transactions on Control of Network Systems}, vol.~7, no.~2, pp.
  758--769, 2019.

\bibitem{Gru18}
V.~Grushkovskaya, A.~Zuyev, and C.~Ebenbauer, ``On a class of generating vector
  fields for the extremum seeking problem: {L}ie bracket approximation and
  stability properties,'' \emph{Automatica}, vol.~94, pp. 151 -- 160, 2018.

\bibitem{Dur13}
H.-B. D{\"u}rr, M.~S. Stankovi{\'c}, C.~Ebenbauer, and K.~H. Johansson, ``Lie
  bracket approximation of extremum seeking systems,'' \emph{Automatica},
  vol.~49, no.~6, pp. 1538--1552, 2013.

\bibitem{Sch14}
A.~Scheinker and M.~Krstić, ``Extremum seeking with bounded update rates,''
  \emph{Systems \& Control Letters}, vol.~63, pp. 25 -- 31, 2014.

\bibitem{Lab22}
C.~Labar, C.~Ebenbauer, and L.~Marconi, ``Iss-like properties in lie-bracket
  approximations and application to extremum seeking,'' \emph{Automatica}, vol.
  136, p. 110041, 2022.

\end{thebibliography}

\end{document}